\documentclass[11pt]{article}
\usepackage{mathrsfs}
\usepackage{amsfonts}
\usepackage{amsmath}
\usepackage{amssymb}
\usepackage{amsopn}
\usepackage{amsthm}
\usepackage{amstext}
\usepackage{color}
\def\ep{\epsilon}
\def\dt{\delta}

\def\lam{\lambda}

\def\la{\langle}
\def\ra{\rangle}

\def\nlam{\lambda_1,\cdots,\lambda_n}
\def\nmu{\mu_1,\cdots,\mu_n}

\def\nla{(\lam_1,\cdots,\lam_n)}
\def\nmu{(\mu_1,\cdots,\mu_n)}

\def\diag{{\rm diag}}

\def\sp{{\rm sp}}

\def\tr{{\rm tr}}
\def\conv{{\rm conv}}

\newcommand{\beq}{\begin{eqnarray*}}
\newcommand{\eneq}{\end{eqnarray*}}
\newcommand{\beqq}{\begin{eqnarray}}
\newcommand{\eneqq}{\end{eqnarray}}
\def\dd{\mathcal{D}}

\def\uu{\mathcal{U}}
\def\nnn{\mathcal{N}}

\def\cc{\mathbb{C}}
\def\rr{\mathbb{R}}

\newtheorem{lemma}{Lemma}[section]
\newtheorem{theorem}[lemma]{Theorem}
\newtheorem{corollary}[lemma]{Corollary}
\newtheorem{proposition}[lemma]{Proposition}
\newtheorem{definition}[lemma]{Definition}

\renewcommand{\phi}{\varphi}

\newcommand{\Z}{\mathbb{Z}}

\newcommand{\R}{\mathbb{R}}
\newcommand{\C}{\mathbb{C}}

\numberwithin{equation}{section}

\newcommand{\Aff}{\operatorname{Aff}}

\newcommand{\Tr}{\operatorname{Tr}}

\newcommand{\cpc}{completely positive contractive linear map}

\newcommand{\hm}{homomorphism}

\newcommand{\andeqn}{\,\,\,{\rm and}\,\,\,}
\newcommand{\rforal}{\,\,\,{\rm for\,\,\,all}\,\,\,}
\newcommand{\CA}{$C^*$-algebra}
\newcommand{\SCA}{$C^*$-subalgebra}
\newcommand{\af}{{\alpha}}

\newcommand{\wilog}{without loss of generality}
\newcommand{\Wlog}{Without loss of generality}

\newcommand{\tforal}{\,\,\,\text{for\,\,\,all}\,\,\,}
\newcommand{\tand}{\,\,\,\text{and}\,\,\,}

\usepackage{amsfonts}
\usepackage{mathrsfs}
\usepackage{textcomp}
\usepackage[all]{xy}

\begin{document}
\date{}
\title{ Convex hulls of unitary orbits of normal elements in $C^*$-algebras with tracial rank zero   }
\author{Shanwen Hu and  Huaxin Lin\\
 Research Center for Operator Algebras,\\
 School of Mathematical Sciences\\
 at East China Normal University,\\
 Shanghai Key Laboratory of PMMP\\
and \\
Department of Mathematics\\
University of Oregon\\
}

\maketitle

\begin{abstract} \noindent
Let $A$ be a unital  separable simple \CA\, with tracial rank zero and let
$x, \, y\in A$ be two normal elements.
We show that $x$ is in the closure of
the convex full of the unitary obit of $y$ if and only if
there exists a sequence of unital  completely positive  linear maps $\phi_n$ from
$A$ to $A$ such that the sequence {{$\phi_n(y)$ convergent  to $x$ in norm }}and also  approximately   preserves the trace values.
 A purely measure theoretical description for normal elements in 
the closure of convex hull of unitary orbit of $y$ is also given.
In the case that $A$ has a unique tracial state some classical results about the closure
of the convex hull of the unitary orbits  in von Neumann
algebras are proved to be hold in \CA s setting.

\vspace{1mm}


\end{abstract}

\section{Introduction}

In  the algebra of $n$ by $n$ matrices over the complex field, two normal elements  are unitarily equivalent
if and only if they have the same eigenvalues counting multiplicities.  One can also describe
the convex hull of unitary orbit of  a given self-adjoint matrix using eigenvalue distribution. Indeed,
by a classical Horn's theorem (\cite{Hor}, see also \cite{TA1}), if $x$ and $y$
are two self-adjoint  matrices, then $x$ is in the convex hull of unitary orbit
of $y$ if  eigenvalues are majorized by those of $x.$

  In infinite dimensional
spaces, one studies the closure of convex hull of unitary orbit of an operator.  These results were
extended to von Neumann algebras (see, for example,
\cite{HN} and \cite{K1}).
Moreover,  in  \cite{H1} (and
\cite{AP}), the
closure of convex hull of unitary orbits
of normal elements in II$_1$ factors were also studied.
Since these description are closely related to the measure theory, via spectral theory,  one may expect that  the original
description in finite  matrix algebras carries out to  von Neumman algebras or at least
{II}$_1$ factors.

The situation is rather different in \CA s since
spectral theory  on longer holds.  However, more recently, in \cite{S}, \cite{PS} and \cite{PS2},
the closure of the convex hull of self-adjoint elements in unital simple \CA s with tracial rank zero
(or real rank zero and stable rank one, as well as other regularities) has been studied.
The current research was inspired by these  researches together with \cite{HL}.

One of the convenience of study  of self-adjoint elements  is that
the \SCA\, generated by a self-adjoint element has certain weak semi-projective property.
Moreover, the assumption that \CA s have real rank zero means that self-adjoint elements
can be approximated by those with finite spectrum.
These advantages disappear   when $x$ and $y$ are only assumed to be normal.

In the current paper,  we study the  normal elements in the closure of convex hull of unitary orbit of  normal elements in unital simple \CA s
with tracial rank zero.  The weak semi-projectivity property  can be partially recovered by the theorem
of \cite{FR}.
However, normal elements in general simple \CA s with tracial rank zero may not be approximated by
normal elements with finite spectrum.  Nevertheless, a theorem in \cite{LinFU} shows that the normal
elements in general simple \CA s with tracial rank zero can actually be approximated by those
with finite spectrum if  a $K_1$-related index vanish.  Moreover, unitary orbits of normal
elements in general simple \CA s with tracial rank zero were characterized in \cite{LinCHD}.
Using these results,   in this paper, we characterize  the normal elements in
the closure of convex hull of unitary orbits of normal elements  in  a general
simple \CA\, with tracial rank zero (see \ref{beta} below).
This keeps the same sprit of results in {II}$_1$-factors as in \cite{H1} and  \cite{AP} even though
the simple \CA\, $A$ may have rich tracial simplex.  On the other hand,
say, if we assume that $A$ also has a unique tracial state,
then a purely measure theoretical  description  of normal elements in
the closure of convex full of normal elements can be presented (see \ref{T5} below).
We also extend the result slightly beyond the case that $A$ has tracial rank zero (see \ref{Trr0}).

Suppose that $x$ is a normal element in the closure of convex hull of the unitary orbit
of  a normal element $y$ and $y$ is in the closure of convex hull of the unitary
orbit of $x.$ Then, in a {II}$_1$-factor $M,$ $x$ and $y$ are approximately unitarily equivalent
(see  Theorem 5.1  of \cite{AP}).
In a general  unital simple \CA\, $A$ with tracial rank zero, this  no longer holds simply because
the presence of non-trivial $K_1$ as well as infinitesimal elements in $K_0(A).$
However, when these $K$-theoretical obstacles disappear, we show
that these two notions still coincide in simple \CA\, of tracial rank zero.
In particular, we show that, in a unital simple AF-algebra $A$ { {with  a unique trace}} if both ${\rm sp}(x)$ and
${\rm sp}(y)$ are connected, and $x$ is in the closure of convex hull of the unitary orbit
of  a normal element $y$ and $y$ is in the closure of convex hull of the unitary
orbit of $x,$ then $x$ and $y$ are approximately unitarily equivalent.

{\bf Acknowledgements}
Much of this research work was done when both authors were in the Research Center of
Operator Algebras at East China Normal University which is partially supported  by Shanghai Key Laboratory of PMMP,  The
Science and Technology Commission of Shanghai Municipality (STCSM), grant \#13dz2260400
  and by  a NNSF grant  (11531003 ).  During the research, the second named author was also supported by
  a NSF grant (DMS 1665183).

\section{Notations}

Let $A$
 be a unital $C^*$-algebra,
We will use the following convention:

(1) $U(A)$ is the  unitary group of $A.$

(2) $\nnn(A)$ is the set of all normal elements of $A,$  $A_{s.a}$ is the set of
all self-adjoint elements of $A$ and $A_+$ is the set of positive elements of $A.$

(3) $\nnn_0(A)=\{x\in\nnn(A):\forall\lam\not\in\sp(x),[\lam-x]=0\,\,{\rm in}\,\,K_1(A)\}.$

(4) For any  $a\in A$, $\uu(a)=\{u^*au:u\in U(A)\}$ is the unitary orbit of $a.$

(5) For any  $a\in A$,  $\conv(\uu(a))$ is the convex hull of the unitary orbit $\uu(a).$

(6) If $p$ is a projection of $A$, $a\in pAp$, $\conv(\uu_p(a))$ is the convex hull of unitary orbit in $pAp$.

(7) $T(A)$ is the set of all tracial states.
If $\tau\in T(A),$ then $\tau\otimes \tr$ is a tracial state of $M_n(A),$ where
 $\tr$ is the tracial state of $M_n(\C).$ We shall continue to use $\tau$ for $\tau\otimes \tr.$

(8) Let  $a,\, b\in A$ and let $\ep>0.$  Let us write $a\approx_{\ep} b$ if $\|a-b\|<\ep.$ Suppose that
$S\subset A$ is  a subset. Let us write $a\in_{\ep} S$ if $\inf\{\|a-s\|: s\in S\}<\ep.$
We may write $a\in_{\ep'} S$ including the case $\ep'=0$ which we mean
that $a\in S.$

(9) Denote by $GL(A)$ the set of invertible elements.
 Recall that   $A$ has stable rank one, if $GL(A)$ is dense in $A.$

(10) Let $p, q\in A$ be  two projections,
We write $[p]=[q]$ if
there exists a  $v\in A$ such that
$v^*v=p$ and $vv^*=q.$ We write $[p]\le [q].$ If $[p]=[q']$ for some projection $q'\le q.$

We write $[p]\le_u [q],$ if there exists a unitary $u\in A$ such that
$u^*pu\le q,$ and $[p]=_u[q],$ if $u^*pu=q.$

If $A$ has stable rank one, then $[p]\le [q]$ is the same as $[p]\le_u [q]$ and
$[p]=[q]$ is the same as $[p]=_u[q].$ Note, almost all the cases in this paper,
$A$ has stable rank one.

(11) Let $x, y\in A_+$ be positive elements.
We write $x\lesssim y,$ if there exists $r_n\in A$ such that
$\lim_{n\to\infty}\|r_n^*yr_n- x\|=0.$
In case $x$ and $y$ are projections, then there exists a partial isometry $v\in A$
such that $v^*v=x$ and $vv^*\le y.$
If $A$ has stable rank one and $x\lesssim y,$ then there exists $z\in A$ such
that $z^*z=x$ and $zz^*\in \overline{yAy}.$

Let $K\ge 1$ be an integer. We write $K\la x\ra \le \la y\ra,$
if there are $K$ mutually orthogonal positive elements $x_1, x_2, ...,x_n\in M_m(A)$
for some $m\ge 1,$ $x_1+x_2+\cdots +x_n\lesssim y$ in
$M_m(A)$ and $x_i\lesssim x$ and $x\lesssim x_i,$ $i=1,2,...,n.$

If $p\in A$ is a projection and $p\lesssim x,$ then there exists partial isometry $v\in A$ such
that $v^*v=p$ and $vv^*\in \overline{xAx}.$

(12) A linear map $\phi: A\to A$ is said to trace preserving if
$\tau\circ \phi=\tau$ for all $\tau\in T(A).$

(13)  Let ${\cal F}\subset A$ be a finite subset and $\ep>0.$
Suppose that $B$ is another \CA. A positive linear map $L: A\to B$ is  said to be
${\cal F}$-$\ep$-multiplicative if $\|L(xy)-L(x)L(y)\|<\ep$ for all $x, y\in {\cal F}.$

\section{Preliminaries}

The following lemma is well-known.
\begin{lemma}\label{lem00}

Suppose $\{a,b,a_i:i=1,...,k\}\subset A$, then
 $a\in_{\ep_1}\conv(\uu(b))$ and $b\in_{\ep_2}\conv(\uu(c))$ imply $$a\in_{\ep_1+\ep_2}\conv(\uu(c)).$$

\end{lemma}
\begin{proof}
 There are
 $
\{u_i,v_j:i=1,...,m;j=1,...,n\}\subset U(A)
$
and
$\{t_i,s_j:i=1,...,m;j=1,...,n\}\subset (0,1)$
 with
$$
\sum_{i=1}^mt_i=1\,\,\,{\rm and}\,\,\, \sum_{j=1}^ns_j=1$$
satisfying  $$a\approx_{\ep_1}\sum_{i=1}^mt_iu^*_ibu_i,
\,\,\,\,{\rm and}\,\,\,\,b\approx_{\ep_2}\sum_{j=1}^ns_jv^*_jcv_j.$$
Let
$$
a'=\sum_{i=1}^m\sum_{j=1}^nt_is_ju^*_iv^*_jcv_ju_i,
$$
then $a'\in\conv(\uu(c))$ and
$$
a'\approx_{\ep_2}\sum_{i=1}^mt_iu^*_ibu_i\approx_{\ep_1}a.
$$

\end{proof}

\begin{definition}\label{def}
Recall that $D\in M_n(\cc)$ is called doubly stochastic matrix if $D=(d_{ij})$ with $d_{ij}\in [0,1]$ with $\sum_i^nd_{ij}=\sum_{j=1}^nd_{ij}=1$ for all $i,j$.
Denote by  $\dd_n$ the
set of all  doubly stochastic matrices in $M_n(\cc).$
\end{definition}

\begin{definition}\label{def1}

For any $x=(\nlam),y=\nmu\in\cc^n$, we write
$x\prec y$
if there is $D=(d_{i,j})\in\dd_n$ such that
\beqq\label{upt}
\begin{pmatrix}
 \lam_1\\\vdots\\\lam_n
\end{pmatrix}
=\left(\begin{array}{ccc}
 d_{11}&\cdots&d_{1n}\\
 \vdots&&\vdots\\
 d_{n1}&\cdots&d_{nn}
\end{array}\right)\begin{pmatrix}
 \mu_1\\\vdots\\\mu_n
\end{pmatrix}.
\eneqq
\end{definition}
We may also  write  $$\nla^T=D\nmu^T$$
instead of \eqref{upt}.
\vspace{0.1in}

The following is a variation of Horn's theorem (see also \cite{TA1}).

\begin{lemma}\label{lemA}
Let  $\nla,\nmu\in\cc^n.$  Then  the following conditions are equivalent:

{\rm (1)} $\nla\prec \nmu$;

{\rm (2)} $\diag\nla\in\conv(\uu(\diag\nmu))$ in $M_n(\cc)$.

{\rm (3)} There is a unital  trace preserving completely positive linear mapping $\Phi$ on $M_n(\cc)$
such that $\Phi(\diag\nmu)=\diag\nla.$

{\rm (4)} There is a unital completely positive linear mapping $\phi$ on
$\cc^n$ such that $\phi\nmu=\nla$ and $\tau\circ d_n(\phi(a))=\tau\circ d_n(a)$
for all $a\in \C^n,$ where $d_n: \C^n\to M_n(\C)$ is defined by
$d_n(\af_1,\af_2,...,\af_n)=\diag(\af_1,\af_2,...,\af_n)$ for all $(\af_1, \af_2,...,\af_n)\in \C^n,$
and $\tau$ is the unique tracial state on $M_n(\C).$
\end{lemma}
\begin{proof}
$(1)\Rightarrow (2)$: If $\nla\prec\nmu$, then there is $D=(d_{ij})\in \dd_n$ such that  $\nla^T=D\nmu^T.$ By Birkhoff's Theorem \cite{GB},
$$
D=\sum_{\sigma\in\Sigma_n}t_\sigma v_\sigma,
$$
where $\Sigma_n$ is the permutation group of $\{1,...,n\}$, $t_\sigma\in [0,1]$ with $\sum_{\sigma\in\Sigma_n}t_\sigma=1$, and $v_\sigma$ is the permutation on $\cc^n$.
One  may  check that
$$\diag\nla=\sum_{\sigma\in\Sigma_n}t_\sigma u^*_\sigma(\diag\nmu)u_\sigma,$$ where $u_\sigma$ is the unitary of $M_n(\cc)$  induced by $v_\sigma$.
That is, viewing element $u_\sigma$ as a linear operator on $\cc^n,$ for any $\nla\in \cc^n$, $u_\sigma(\nla)=(\lam_{\sigma_1},...,\lam_{\sigma_n})$.

$(2)\Rightarrow (3)$: If $x=\sum_{i=1}^Nt_iu^*_iyu_i$,
define $\Phi(z)=\sum_{i=1}^Nt_iu^*_i(z)u_i$ for any $z\in M_n(\cc)$.
Then $\Phi$ is a trace preserving \cpc.

$(3)\Rightarrow (4)$:
Let $E$ be the map from $M_n(\cc)$ to $\C^n$ be defined by
$E(a_{ij})=(a_{11},a_{22},...,a_{nn})$ for all $a=(a_{ij})\in M_n(\cc).$
For any $z=(\eta_1,...,\eta_n)\in\cc^n$, define
$$\phi(z)=E(\Phi(z))\in   \C^n.$$
Then one checks that $\phi$ meets the requirements.

$(4)\Rightarrow (1)$:
Fix
any $\nmu\in\cc^n.$
One divides $\{1,2,...,n\}$ into
$N$ disjoint subsets
$S_k\subset \{1,2,...,n\},k=1,2,...,N$
such that
 $\sqcup S_k=\{1,2,...,n\}$
  and $i\in S_k,$ if and only if $\mu_i=\mu_k.$

  Denote by $e_i=\diag(\overbrace{0,...,0}^{i-1}, 1, 0,...,0),$ $i=1,2,...,n.$
  Set
  $Q_k=\sum_{i\in S_k}e_i,(k=1,2,...,N).$
  and
  $C_d=C^*(Q_1,...,Q_N)$ in $\cc^n$.
  We now view ${\tilde \phi}: d_n(\C^n)\to  d_n(\C^n)$ as
  ${\tilde \phi}(\diag(\af_1, \af_2, ...,\af_n))=d_n(\phi(\af_1,\af_2,...,\af_n))$
  for all $(\af_1, \af_2,...,\af_n)\in \C^n.$

For any $i\in S_k,$ ($k=1,2,...,N$), define
$$
d_{ij}=\frac{\Tr(e_j{\tilde \phi}(Q_k)e_j))}{|S_k|},\,\,j=1,2,...,n,$$
where $\Tr$ is the non-normalized trace on $M_n(\C).$
Then $(d_{ij})\in\dd_n$. In fact,
\beq
&&\sum_{i=1}^nd_{ij}=\sum_{k=1}^N\sum_{i\in S_k}\frac{\Tr(e_j{\tilde \phi}(Q_k))}{|S_k|}
=\Tr(e_j{\tilde\phi}(1_{\cc^n}))=1,j=1,2,...,n;\\
&&\sum_{j=1}^nd_{ij}
=\sum_{j=1}^n\frac{\Tr(e_j{\tilde\phi}(Q_k))}{|S_k|}
=\frac{\Tr({\tilde \phi}(Q_k))}{|S_k|}=1,i\in S_k,k=1,2,..,N.
\eneq
Since $\phi(y)=x, $  $\nla^T=D\nmu^T.$ 

\end{proof}

\begin{corollary}\label{corB} Suppose $A$ is a $C^*$-algebra.
If $\{p_i:i=1,...,n\}$ and $\{q_i:i=1,...,n\}$ are $n$-tuples of mutually orthogonal projections in $A$ with
$[p_i]=_u[q_j]$ for any $i,j$, and if
$$
x=\sum_{i=1}^{n}\lambda_{i}p_{i},\quad\quad y=\sum_{i=1}^{n}\mu_{i}q_{i},
$$
then

(1) $x\in\conv(\uu(y))$ if and only if $\nla\prec \nmu$.

(2) Let $B=\{\sum_{i=1}^n\lam_ip_i:\lam\in\cc\}$ and $p=\sum_{i=1}^np_i$.
If there is $D\in\dd_n$ such that
$\nla^T=D\nmu,$ then
 there is a unital trace preserving completely positive linear map $\phi_D$ on $pAp$  which maps
$B$ into $B$ such that

$$
\phi_D(x)=\sum_{i=1}^{n}\mu_{i}p_{i}.
$$

\end{corollary}

\begin{proof}
There exists $u\in U(A)$ such that $u^*p_iu=q_i,i=1,...,n$.
We  have $$x=\sum_{i=1}^n\lam_ip_i=u^*\Big(\sum_{i=1}^n\lam_iq_i\Big)u.$$
Therefore
$
x\in\conv(\uu(y))$ if and only if $ \sum_{i=1}^n\lam_iq_i\in\conv(\uu(y)).
$

Thus $(1)$ follows from $(1)\Leftrightarrow (2)$ in   \ref{lemA} immediately.

(2)  Since $pAp$ is isomorphic to $M_n(p_1Ap_1)$. $U(M_n(\cc))$ can be viewed as an subset $U(pAp)$.
For any $D\in\dd_n$, by $(1)\Leftrightarrow (4)$ in \ref{lemA},
$\phi_D=\sum_{\sigma}t_\sigma Ad(u_\sigma)$, defined on $M_n(\cc)$,
can be extended to $pAp$.

\end{proof}


\begin{lemma}\label{Ln25}
Suppose $A$ is a unital \CA.

(1) If $K>1$ is an integer and
$\{e_i,p_{i}:i=1,...,l\}$
are mutually orthogonal  projections of $A$ with $(K+2)[e_i]=_u[p_i],i=1,2,...,l,$
and if
$$x=\sum_{i=1}^l\lam_ip_i\in pAp,\,\,\,\,\,
x'=\sum_{i=1}^l\lam_ie_i\in eAe,$$
then
\beq
x+x'\in_{\ep_1}\conv(\uu_{p+e}(x)),
\eneq
where $p=\sum_{i=1}^lp_i,e=\sum_{i=1}^le_i$ and $\ep_1=\frac{\|x\|}{K}.$

(2) If $q,e,p$ are mutually orthogonal projections in $A$ with $q+e+p=1_A$ and
$K[e]\le_u[q]$.
Then, for any  $y'\in eAe$ and any
$y\in \nnn (pAp),$
one has that
\beq
y\in_{\ep_2}\conv(\uu(y'+y)),
\eneq
where  $\ep_2={\frac{\|y'\|}{K+1}}$.

(3) If $e,p$ are mutually orthogonal projections with $e+p=1_A$, $\{p_i,e_i,i=1,2,...,l\}$ are projections with $\sum_{i=1}^lp_i=p,\sum_{i=1}^le_i=e$ and $(K+2)[e_i]\le_u [p_i],i=1,2,...,l,$
where $K>1,$ and if  $x=\sum_{i=1}^l\lam_ip_i$ and $x'=\sum_{i=1}^l\lam_ie_i$, then for any
$y'\in eAe$,
\beq
x+x'\in_{\ep_3}\conv(\uu(x+y')),
\eneq
where $\ep_3=\frac{\|y'\|+3\|x\|}{K}.$
\end{lemma}
\begin{proof}
 (1)  We first claim the following:

   If
   $$
   y=\diag(\overbrace{a,...,a}^K)\andeqn y'=\diag(0, \overbrace{a,a,..a}^{K-1}),
   $$
   then
   \beqq\label{Ln-1}
  {\rm dist}(y, \conv(\uu(y')))\le \frac{\|a\|}{K}.
   \eneqq
   There are unitaries $u_j\in A$ such that
   \beq
   u_j^*y'u_j=\diag(\overbrace{a, a,...,a}^{j}, 0, a,..,a),\,\,\, j=1,2,..,K-1.
   \eneq
 Then
 \beq
 (1/K)y'+\sum_{j=1}^{K-1}(1/K)u_j^*y'u_j=\frac{K-1}{K}\diag(\overbrace{a,a,...,a}^K).
 \eneq
 Note that
 \beq
 \|y-\frac{K-1}{K}\diag(\overbrace{a,a,...,a}^K)\|=\|a\|/K.
 \eneq
 This proves the claim.

There are mutually orthogonal projections $p_{j,i}:1\le j\le l,1\le i\le K+2$ such that $p_i=\sum_{j=1}^{l}p_{j,i}$ and $[p_{j,i}]=_u[e_i],i=1,2,...,K+2$.
  There is a unitary $u\in (p+e)A(p+e)$  such that
 \beqq\label{n25-1}
 u^*xu=x'\oplus \sum_{j=1}^l \lambda_jp_j',
 \eneqq
where $p_j'=\sum_{i=2}^{K+2}p_{j,i}.$ Put $x_1=\sum_{j=1}^l \lambda_jp_j'.$
Note that $\lambda_jp_j'$ may be written as
\beq
\lambda_jp_j'=\diag(\overbrace{\lambda_j,\lambda_j,...,\lam_j}^{K})\oplus \lam_j p_{j,K+2}.
\eneq
By the claim above, ${\rm dist}(x, \conv(\uu(x_1)))\le \ep_1.$ It follows from this, \eqref{n25-1}  and \ref{lem00} that
$$
{\rm dist}(x+x', \conv(\uu(x))\le \ep_1.
$$

(2) There are unitaries $u_j\in A$ and mutually { {orthogonal}} projections $q_j\in qAq$ such that
$u_j^*eu_j=q_j,$ $j=1,2,...,K.$
Let $t_j=\frac 1{K+1},$ $j=0, 1,2,...,K,$
define
\beq
y_0= t_0y'+ \sum_{j=1}^K t_ju_j^*y'u_j.
\eneq
Then $y_0\in {\rm conv}({\cal U}_{q+e}(y')).$ Moreover. $\|y_0\|\le \frac{\|y'\|}{K+1}.$
It follows that
\beq
0\in_{\ep_2} \conv(\uu_{q+e}(y')).
\eneq
Since $y\in\conv(\uu_{p}(y))$, by \ref{lem00}, we have
\beq
y\in_{\ep_2}\conv(\uu(y+y')).
\eneq

(3) First, we consider the case  $(K+2)[e_i]=_u[p_i].$

If  $0\in \sp(x),$ without loss generality,
we may assume $\lam_1=0$. Let $q=p_1, p'=\sum_{i=2}^lp_i$, then  by (1) (where  $p$
is replaced by $p'$),  $x+x'\in_{\frac{\|x\|} {K}}\conv(\uu_{e+p'}(x))$. Applying (2),
$$x\in_{\eta_1}\conv(\uu_{q+p'+e}(x+y')),\,\,\,\,\,{\rm where}\,\,\,\eta_1=\frac {\|y'\|}{K+1}.$$
So by  \ref{lem00},
$$x+x'\in_{\eta_2}\conv(\uu(x+y')),\,\,\,\,\eta_2=\frac {\|x\|+\|y'\|}{K}.$$

 In case $0\not\in \sp(x),$  we consider $x-\lam_1$ and $y'-\lam_1.$ Then  $0\in\sp(x-\lam_1).$  Replacing $x,y'$ by $x-\lam_1,y'-\lam_1$, by  the proof above, we have
$$x+x'-\lam_1\in_{\eta_3}\conv(\uu(x+y'-\lam_1)),\,\,\,\,\eta_3=
\frac{\|y'-\lam_1\|+\|x-\lam_1\|}{K}.$$
Therefore,
$$x+x'\in_{\ep_3}\conv(\uu(x+y')),\,\,\,\,{\rm where}\,\,\,\ep_3=\frac{\|y'\|+3\|x\|}{K}.$$

In general,
let projection $p'_i\le p_i$ with $(K+2)[e_i]=[p'_i]$ and $x_1=\sum_{i=1}^l\lam_ip'_i.$ Then from what has been proved,
$x_1+x'\in_{\ep_3}\conv(\uu(x_1+y'))$.
It then follows from  \ref{lem00} and the fact $x+x'=(x_1+x')\oplus(x-x_1)$ that
$x+x'\in_{\ep_3}\conv(\uu(x+y')).$
\end{proof}

\begin{proposition}\label{PPP}
Let $A$ be a unital \CA, $x$ and $y$ be two normal elements.
Suppose that $x\in \overline{\conv(\uu(y))}.$ Then
there exists a sequence of  trace preserving \cpc s $\Phi_n: A\to A$ such that
$\lim_{n\to\infty}\|\Phi_n(y)-x\|=0.$
\end{proposition}

\begin{proof}
There are $0\le \lambda_{i,n}\le 1$ with
$\sum_{i=1}^{r(n)} \lambda_{i,n}=1$ and
unitaries $u_{i,n}\in A$ such that
\beq
\lim_{n\to\infty}\|x-\sum_{i=1}^{r(n)}\lambda_{i,n} (u_{i,n}^* yu_{i,n})\|=0.
\eneq
Define $\Phi_n: A\to A$ by
$\Phi_n(a)=\sum_{i=1}^{r(n)} \lambda_{i,n}u_{i,n}^* au_{i,n}$ for all $a\in A.$
Then $\psi_n$ is a unital \cpc\, and
\beq
\tau(\Phi_n(a))=\sum_{i=1}^{r(n)}\lambda_{i,n}\tau(u_{i,n}^*au_{i,n})
=\sum_{i=1}^{r(n)}\lam_{i,n}\tau(a)=\tau(a)
\eneq
for all $a\in A.$

\end{proof}

The following is known. We state here for reader's convenience.

\begin{lemma}\label{Lcp}
Let $A$ be a unital simple \CA\, with $T(A)\not=\emptyset.$
Let $a, b\in A_+$ with $\|a\|, \|b\|\le 1$ such that $\tau(a)<\tau(b)$ for all $\tau\in T(A).$
Then there is $0\le b_0\le b$ such that $\tau(b_0)=\tau(a)$ for all $\tau\in T(A).$
Moreover, there are $y_n\in A$ such that
\beq
\sum_{n=1}^{\infty}y_n^*y_n=a\andeqn \sum_{n=1}^{\infty}y_ny_n^*=b_0\le b,
\eneq
where the sums converge in norm.
\end{lemma}

\begin{proof}
Let $f\in \Aff(T(A))_+$ be such that $f(\tau)=\tau(b-a)$ for all $\tau\in T(A).$
Let $1>\ep>0.$
It follows from  9.3 of \cite{Lincltr1}   that there exists $0\le b'\le 1+\ep$ in $A$ such that
$\tau(b')=f(\tau)$ for all $\tau\in T(A).$
Put $b_1=\diag(a,b')$ in $M_2(A)$ and put $B=M_2(A).$ Then
$\tau(b_1)=\tau(b)$ for all $\tau\in T(B).$  By Theorem 2.9  of \cite{CP},
$b_1-b\in A_0$ (notation in \cite{CP}). In other words, there are $x_1,x_2,...,$
in $B$ such that
\beq
\sum_{n=1}^{\infty}x_n^*x_n=b_1\andeqn \sum_{n=1}^{\infty} x_nx_n^*=b.
\eneq
Let $e_1=\diag(1_A,0).$ Then $e_1a=ae_1=a.$ Put $y_n=x_ne_1.$
Then
\beq
\sum_{n=1}^{\infty}y_n^*y_n=e_1(\sum_{n=1}^{\infty}x_n^*x_n)e_1=a\andeqn\\
\sum_{n=1}^{\infty}y_ny_n^*=\sum_{n=1}^{\infty}x_ne_1x_n^*\le \sum_{n=1}^{\infty}x_nx_n^*=b.
\eneq
Choose $b_0=\sum_{n=1}^{\infty}y_ny_n^*.$ Then
${ 0\le } b_0\le  b$ and $\tau(b_0)=\tau(a)$ for all $\tau\in T(A).$
\end{proof}

\section{
Simple $C^*$-algebras with tracial rank zero}

\begin{definition}[\cite{Linlondon}]\label{def2}
A unital simple $C^*$-algebra has tracial rank zero  if
for any $\ep>0$,  any non-zero $r\in A_+$,  any ${\cal F}=\{x_1,...,x_n\}\subset A$, there exists
a finite dimensional $C^*$-algebra $B\subset A$ with unit $p\in A$, and such that
\beq
&&x\approx_\ep x'+x'', \,\,\,where\,\,\, x'\in B,\\
&&x''\in (1-p)A(1-p),\,\,\,for\,\, all\,\,\, x\in {\cal F},\\
&&1-p\lesssim r.
\eneq
If $A$ has tracial rank zero, we write $TR(A)=0.$

The above definition  is  equivalent to the following:
For any $\ep>0$ and any finite subset ${\cal F}\subset A$  and any $r\in A_+\setminus \{0\},$
there
exists a projection $p\in A,$ a finite dimensional \SCA\, $F$ of $A$ with $1_F=p,$ and
a unital  ${\cal F}$-$\ep$-multiplicative  \cpc\, $L: A\to F$ such that
\beq
&&\|px-xp\|<\ep\rforal x\in {\cal F},\\
&&{\rm dist}(pxp, F)<\ep\tforal x\in {\cal F},\\
&&\|c-(pcp+L(c))\|<\ep\rforal x\in {\cal F}\andeqn\\
&&1-p\lesssim r.
\eneq

If $TR(A)=0,$ then $A$ has stable rank one and real rank zero.
Moreover, if $p, q\in M_n(A)$ are two projections for some integer $n\ge 1$
and $\tau(p)<\tau(q)$ for all $\tau\in T(A),$ then $[p]\le [q]$ {\rm (}see  {\rm \cite{Linlondon})}.

\end{definition}

\begin{lemma}\label{Csec3for4}
{{Suppose $A$ is a unital simple $C^*$-algebra with tracial rank zero. If $x,y\in\nnn(A)$,
then
for any $\ep>0$, any  integer $K\ge 1$,  and any nonzero projection $r\in A,$  there exists
a finite dimensional $C^*$-algebra $B$ of $A$ with unit $p$ such that}}

 $x\approx_\ep x'+x'',y\approx_\ep y'+y''$, where $x',y'\in \nnn(B), x'',y''\in \nnn((1-p)A(1-p))$,
 \beqq\label{lem11}
x'=\sum_{i=1}^l\lam_{i}p_{j}\,\,\,\,\,\,\,
y'=\sum_{i=1}^m\mu_iq_i,
\eneqq
where $\{\lam_1, \lam_2,...,\lam_l\}$ is $\ep$-dense in $\sp(x),$
$\{\mu_1, \mu_2,...,\mu_m\}$ is $\ep$-dense in $\sp(y),$
and
\beqq\label{lem13}
K[1-p]\le [r],\,\,\, K[1-p]\le [p_i]\andeqn  K[1-p]\le [q_j],
\eneqq
$i=1,...,l, \,\, j=1,2,...,m.$
Moreover, we may also assume that $\sp(x'')$ is $\ep$-dense in $\sp(x)$ and
$\sp(y'')$ is $\ep$-dense in $\sp(y).$

\end{lemma}

\begin{proof}
\Wlog, we may assume that $\|x\|,\|y\|\le 1.$
Note that every unital hereditary \SCA\, of $A$ has stable rank one
{\cite {Lin3}}.  In particular, they have (IR) (see \cite{FR}).
Fix $\ep>0.$ Then
fix $0<\ep_0 <\ep/4.$ By 4.4 of \cite{FR}, there exists $\dt>0$ such that
for any $z$ in a \CA\,    with (IR)   with $\|z\|\le 1$ and  with the property that
\beq
\|z^*z-zz^*\|<\dt,
\eneq
then there exists a normal element $z_N$ in that \CA\, such that
\beq
\|z-z_N\|<\ep_0/2.
\eneq
Let $\{\lambda_1,\lambda_2,...,\lambda_l\}\subset \sp(x)$ and let
$\{\mu_1,\mu_2,...,\mu_m\}\subset \sp(y)$ be such that both sets are $\ep/2$-dense in
$\sp(x)$ and $\sp(y),$ respectively.

There are $f_i\in C(\sp(x))$ and $g_j\in C(\sp(y))$ such that
$f_i(t)=1$ for $|t-\lambda_i|<\ep/4,$ $f_i(t)=0$ for $|t-\lambda_i|>\ep/2$ and
$0\le f_i(t)\le 1;$
$g_j(t)=1$ for $|t-\mu_j|<\ep/4,$ and $g_j(t)=0$ for $|t-\mu_j|>\ep/2,$ $1\le i\le l$ and $1\le j\le m.$
Note that $f_i(x)\not=0$ and $g_j(y)\not=0.$
Since $A$ is simple, there are $a_{i,k}, b_{j, k}\in A$ such that
\beq
\sum_{k=1}^{n(i)}a_{i,k}^*f_i(x)a_{i,k}=1_A\andeqn \sum_{k=1}^{n'(j)}b_{j,k}^*g_j(y)b_{j,k}=1_A.
\eneq
Let $M=\max\{n(i)\|a_{i,k}\|,  n(j)'\|b_{j,k}\|: i,j,k,\}.$
Choose  $\eta=\min\{\ep_0, \dt\}/4(M+1).$
Put ${\cal G}={\cal F}\cup \{x, y, a_{i,k}, b_{j,k}, f_i(x), g_j(x): i,j,k\}.$
Since $TR(A)=0,$  there is a finite dimensional \SCA\, $B'\subset A$ with
$p'=1_{B'}$ such that
\beq
&& B'=M_{r_1}(\C)\bigoplus M_{r_2}(\C)\bigoplus \cdots \bigoplus M_{r_N}(\C),\\
&&\|z-(z'+z'')\|<\eta\rforal z\in {\cal G},\\
&&\|x-(x_1+x_2)\|<\eta,\,\,\, \|y-(y_1+y_2)\|<\eta\andeqn\\
&&{{K}}
[ 1-p'] \le  [r],
\eneq
where $x_1, y_1, z'\in B',$ $x_2, y_2, z''\in (1-p')A(1-p').$
Moreover, we may assume, \wilog, that $\|x_i\|\le 1$ and $\|y_i\|\le 1,$ $i=1,2.$
Note that, since $x$ and $y$ are normal,
\beq
\|x_i^*x_i-x_ix_i^*\|<\dt \andeqn \|y_i^*y_i-y_iy_i^*\|<\dt,\,\,\,i=1,2.
\eneq
It follows from 4.4 of \cite{FR} that
there are $x_0', y_0'\in {\cal N}(B')$ and $x_0'', y_0''\in {\cal N}(A)$ such that
\beq
\|x_0-x_0'\|<\ep_0/2,\,\, \, \|x_0-x_0''\|<\ep_0/2,\,\,\, \|y_0-y_0'\|<\ep_0/2\andeqn\|y_0-y_0''\|<\ep_0/2.
\eneq
Then
\beq
\|x-(x_0'+x_0'')\|<\ep\andeqn \|y-(y_0'+y_0'')\|<\ep.
\eneq
Furthermore, if  $\ep_0$ and $\eta$ are   small enough,  \wilog,   we may assume that
\beq
p'=\sum_{k=1}^{n(i)}(a_{ik}')^*f_i(x_0')a_{ik}',
\,\,\,1-p'=\sum_{k=1}^{{{n(i)}}}
(a_{ik}'')^*f_i(x_0'')a_{ik}'',i=1,2,...,l,\\
p'=\sum_{k=1}^{n'(j)}(b_{jk}')^*g_i(y_0')b_{jk}'\andeqn  1-p'=\sum_{k=1}^{{{n'(j)}}}
(b_{jk}'')^*g_j(y_0'')b_{jk}'',j=1,2...,m.
\eneq
Therefore, we may further assume
that
\beq
x_0'=\sum_{i=1}^l \lambda_i p_i', \,\,\, y_0'=\sum_{j=1}^m \mu_j q_j',
\eneq
where $\pi_k(p_i')\not=0,$ $\pi_k(q_j')\not=0,$ and
{{$\sp(x_0''),\sp(y_0'')$ are $\ep/2$-dense in $\sp(x),\sp(y)$ respectively.}}

Since $A$ is a simple \CA\, of real rank zero, it is easy to find a nonzero
projection $e_0\lesssim p_i'$ and $e_0\lesssim q_j'$ for all $i$ and $j.$

Since
$TR((1-p')A(1-p'))=0$,  {{by repeating the above process, one obtains}}
a finite dimensional $C^*$ subalgebra
$$B''=M_{r_{N+1}}(\cc)\oplus\cdots\oplus M_{r_{k(m)}}(\cc)$$ in $(1-p')A(1-p')$ with unit {{$p''$}}  
such that

(i) $x_0''\approx_{\ep/2} x''_{1}+x''_2,
y_0''\approx_{\ep/2} y''_1+y''_2
$,
where $x''_1,y''_1\in \nnn({{B''}}),x''_2,y''_2\in \nnn(((1-p'-{{p''}})
A(1-p'-{{p''}}
))).$

(ii) $x_1''=\sum_{i=1}^{l'} {{\lambda'_i}}
p_i'',$ $y_1''=\sum_{j=1}^{m'}{{\mu'_j}}
 q_j'',$
where $\{p_1'', p_2'',...,p_{{ {l'}}}''
\}$ and $\{q_1'',q_2'',...,q_{{{m'}}}''
\}$ are two sets of mutually
orthogonal and mutually equivalent projections in $B''.$
{{Let $\{1,2,...,l'\}=\sqcup_{i=1}^l S_i$, where $S_i\subset\{k:|\lam_k'-\lam_i|<\ep/2\}$, then let $\sum_{k\in S_i}\lam_k'p_k''$ replaced by $\lam_i\sum_{k\in S_i}p_k$, without lost generality, we may assume $l'=l,m'=m,\lam'_i=\lam_i,\mu'_j=\mu_j.$}}

(iii) { {$\sp(x''_2),\sp(y''_2)$ are $\ep/2$-dense in $\sp(x''_0),\sp(y''_0)$ respectively,}}
 hence $\sp(x''_2),\sp(y''_2)$ are $\ep$-dense in $\sp(x),\sp(y)$ respectively.

(iv) ${{K}}[1-p'-{ {p''}}
]\le [e_0],$ and,

(v) $\pi_k(p_i'')\not=0$ and $\pi_k'(q_j'')\not=0$ for $\pi_k: B''\to M_{r_k}(\C)$
for all $k=N+1, N+2,...,k(m).$

 Let $B=B'\oplus B''$, $p:=1_B={{p'+p''}}
 ,$   and let $x'=x_0'+x_1'',$
 $y'=y_0'+y_1'',$ $x''=x_2''$ and $y''=y_2''.$
 Using rank one projections in each $M_{r_k},$ we may write
 \beq
 x'=\sum_{i=1}^l\lambda_i p_i\andeqn y'=\sum_{j=1}^m\mu_i q_j
 \eneq
 Then
\beq
K[1-p]\le [r], K[1-p]\le [p_i],\andeqn K[1-p]\le [q_j],
\eneq
for $ i=1,2,...,l,\,j=1,2,...,m.$

\end{proof}

\begin{lemma}\label{Ln43}
Let $A$ be a unital \CA\, with a unique tracial state $\tau$ and let
$\ep_1$  and $\ep_2$ be two positive numbers.
Suppose that
 $\phi:A\to A$  is  a unital completely positive linear map,
$\{p_1,p_2,...,p_{l}\}$ and $\{q_1, q_2, ...,q_{l}\}$ are two sets of  mutually orthogonally projections  with $\sum_{i=1}^l p_i=1_A,$ and there is a unitary $u\in A$ such that
 $u^*p_iu=q_i,$ $i=1,2,...,l.$

 Let
$x=\sum_{i=1}^{l}\lam_ip_i$   and $y=\sum_{i=1}^{l}\mu_iq_i$ be in $\nnn(A),$
where $\lam_i, \mu_i\in \C.$
 Suppose that $S_1, S_2,...,S_N$ are mutually disjoint subsets of $\{1,2,...,l\}$
 such that
 $\sqcup S_k=\{1,2,...,l\}$ and $\mu_i=\mu_k'$ for all $i\in S_k$
 (with $\mu_k'\in S_k$) and
  $y=\sum_{k=1}^N\mu_k' Q_k,$  where $Q_k=\sum_{i\in S_k} q_i.$    If
   $\phi(y)\approx_{\ep_1} x$ and
 \beqq\label{amenal}
 |\tau(\phi(Q_k))-\tau(Q_k)|\le s\ep_2,
 k=1,...,N,
 \eneqq
 where
 \beq
 s=\inf\{\tau(Q_k):k=1,2,...,N;\tau\in T(A)\}.
 \eneq
 Then there is a trace preserving unital completely positive map from $A$ to $A$ such that
\beq
\|\psi(y)-x\|\le 2 \ep_2\|y\|+3\ep_1\andeqn x\in_{2 \ep_2\|y\|+3\ep_1}\conv(\uu(y)).
\eneq
\end{lemma}

\begin{proof}
We first consider the case that $p_i=q_i,$ $i=1,2,...,l.$
 Let  $\phi': A\to A$ be defined by
 \beq
 \phi'(a)&=&\sum_{i=1}^{l}p_i\phi(a)p_i\rforal a\in A.
 \eneq
  Then $\phi'$ is a unital completely positive linear map from $A$ to $A.$
  Moreover, one checks that   $\tau(\phi'(a))=\tau(\phi(a))$
  for any $a\in A.$

 Define, for $i\in S_k,$
 \beq
 d_{i,j}=\tau(p_j\phi(Q_k)p_j)/\tau(Q_k),\,\,j=1,2,...,l,k=1,2,...,N.                                 
  \eneq
Note that
\beq
\sum_{i\in S_k} d_{i,j}=l\tau(p_j\phi(Q_k)),\,\,\, j=1,2,...,l,k=1,...,N.
\eneq

Since $\phi$ is unital, for any $j\in\{1,2,...,l\}$,
\beqq
\sum_{i=1}^{l}d_{i,j}
&=&l\sum_{k=1}^N\sum_{i\in S_k} \tau(p_j\phi(Q_k))/\tau(Q_{k})=1.
\label{Fsumdi}
\eneqq

Since $\|\phi(y)-x\|<\ep_1$ (note that we have assume that $q_i=p_i,$ $i=1,2,...,l$),
for any $j\in\{1,2,...,l\}$,
\beqq\label{43n-08-10}
\|\lambda_jp_j-\sum_{i=1}^{l}\mu_ip_j\phi(p_i)p_j\|
=\|p_j(x-\phi(y))p_j\|<\ep_1.
\eneqq
This also implies that
\beq
\|x-\phi'(y)\|<{\ep_1}.
\eneq
It follows that for any $j=1,2,...,l$,
\beqq\label{43n-08-11}
|\lambda_j-\sum_{i=1}^{l} \mu_id_{i,j}|
&=&l|\tau(\lambda_jp_j-\sum_{i=1}^{l}\mu_ip_j(\phi(p_i)p_j))|<\ep_1,
\eneqq

We also have, for any $i\in S_k,$
\beqq
\sum_{j=1}^{l} d_{i,j}
&=&\sum_{j=1}^{l}\tau(p_j\phi(Q_{ k}))/\tau(Q_{ k})\\
&=&\tau(\phi(Q_{k}))/\tau(Q_{ k}).
\eneqq
By \eqref{amenal}, { for any $k\in\{1,2,...,N\}$, }
\beq
|\tau(\phi(Q_k))/\tau(Q_k)-1|<s\ep_2/\tau(Q_k)\le \ep_2.
\eneq

Then there are  $\ep_i'\in\rr$ with $|\ep_i'|\le \ep_2$ such that
\beqq\label{neww}\sum_{j=1}^{l}d_{i,j}=1+\ep_i',i=1,...,l.\eneqq
Let $\Lambda_+=\{i:\ep_i\ge 0\}$ and $\Lambda_-=\{i:\ep_i<0\}$.
Note that
\beqq\label{newwww}
\sum_{i\in \Lambda_+}(1+\ep_i')+\sum_{i\in \Lambda_-}(1+\ep_i')=\sum_{i=1}^{l} \sum_{j=1}^{l} d_{i,j}=\sum_{j=1}^{l}\sum_{i=1}^{l} d_{i,j}=l.
\eneqq
Therefore
\beqq\label{aaa-n-1}
\sum_{i\in \Lambda_+}\ep_i'+\sum_{i\in \Lambda_-} \ep_i'=0.
\eneqq
Let
\beq
\ep_{i,j}=\frac{d_{i,j}\ep_i'}{1+\ep_i'},\,\,i\in \Lambda_+,\,\,j=1,...,l,
\eneq
then
\beqq\label{a1}
0\le\ep_{i,j}\le d_{i,j},i\in \Lambda_+;\,\,j=1,...,l,
\eneqq
\beqq\label{a11}
\sum_{j=1}^{l} \ep_{i,j}=\ep_i',\,\,\, i\in \Lambda_+,
\eneqq
and, by \eqref{Fsumdi},  for all $j,$
\beqq
\sum_{i\in \Lambda_+}\ep_{i,j}
&\le & \sum_{i=1}^{l}\frac{d_{i,j}\ep_i'}{1+\ep_i'}\\
&\le & \max_{1\le i\le l} \frac{\ep_i'}{1+\ep_i'}\Big(\sum_{i=1}^{l}d_{i,j}\Big)\\\label{a2}
&=&\max_{1\le i\le l} \frac{\ep_i'}{1+\ep_i'}<\ep_2,
\eneqq

Let $a_i=-\ep_i',\,i\in \Lambda_-$ and $b_j=\sum_{i\in \Lambda_+} \ep_{i,j},$ $j\in \{1,2,...,l\}$.

Then, by \eqref{aaa-n-1} and \eqref{a11},
\beq
\sum_{i\in \Lambda_-}a_i=\sum_{i\in \Lambda_-}(-\ep_i')=\sum_{i\in \Lambda_+}\ep_i'=\sum_{j=1}^{l}b_j.
\eneq

By the Reisz interpolation, (see page 85 of \cite {Alf}), there are $$\{\ep_{i,j}:i\in \Lambda_-;\,j=1,2,...,l\}$$
with $\ep_{i,j}\ge 0$ such that
\beqq\label{neww2}
&&\sum_{j=1}^{l}\ep_{i,j}=-\ep_i'= a_i,i\in \Lambda_-\andeqn\\\label{neww1}
&& \sum_{i\in \Lambda_-}\ep_{i,j}=b_j,\,\,\, j\in \{1,2,...,l\}.
\eneqq
Note that
\beqq\label{an3}
\sum_{j=1}^{l} \ep_{i,j}=\ep'_{i}<\ep_2,\,\,\, i\in \Lambda_+.
\eneqq

Put
\beq
&&d'_{i,j}=d_{i,j}-\ep_{i,j},
\,i\in \Lambda_+;j=1,...,l,\\
&&d'_{i,j}=d_{i,j}+\ep_{i,j},i\in \Lambda_-;j=1,2,...,l.
\eneq
By \eqref{a1}, $d_{i,j}'\ge 0.$
Then by \eqref{Fsumdi} and \eqref{neww},
\beq
&&\sum_{i=1}^{l}d'_{i,j}=1-\sum_{i\in \Lambda_+} \ep_{i,j}+\sum_{i\in \Lambda_-} \ep_{i,j} =1,\rforal j=1,2,...,l\,\,\,\hspace{0.3in} ({\rm using} \eqref{neww1})\\
&&\sum_{j=1}^{l}d_{i,j}'=1+\ep_i'-\sum_{j=1}^{l}\ep_{i,j}=1+\ep'_i-\ep'_i=1,\rforal i\in \Lambda_+\,\,\,\,\,\hspace{0.3in}({\rm using}\eqref{a11})\\
&&\sum_{j=1}^{l}d_{i,j}'=1+\ep_i'+\sum_{j=1}^{l}\ep_{i,j}=1+\ep_i'-\ep_i'=1,\rforal i\in \Lambda_-\,\,\,\,\,\,\hspace{0.23in}({\rm using}\eqref{neww2})\\
\eneq
In other words, for any $i,j$,
$$\sum_{i=1}^{l}d'_{i,j}=\sum_{j=1}^{l}d'_{i,j}=1.$$

Let $D'=(d_{i,j}')_{l\times l}$ and let $\phi_{D'}=\sum_{\sigma\in \Sigma_l}t_\sigma {\rm Ad}\, u_\sigma$
be induced  trace preserving completely positive linear maps in part (4) of Lemma \ref{lemA}.  View each $u_\sigma$ as a  unitary matrix in $M_l(\C \cdot 1_A)\subset
M_{l}(A),$ and
define
$$
\phi_{D'}\Big((c_{ij})\Big)\:= \sum_{\sigma\in \Sigma_l}t_\sigma  u_\sigma^*(c_{ij})u_\sigma
$$
for all $(c_{i,j})\in M_{l}(A).$
Note that,
\beq
\phi_{D'}\Big(\sum_{i=1}^{l} \mu_i p_i\Big)=\sum_{i=1}^{l}\lambda_i'p_i,
\eneq
where $(\lambda_1',\lambda_2',...,\lambda'_{l})$ is given by
\beq
(\lambda_1',\lambda_2',...,\lambda'_{l})^T=D'(\mu_1, \mu_2,...,\mu_{l})^T
\eneq

It follows from \eqref{43n-08-11}, \eqref{an3}    that {for any $i=1,2,...,l$,}
\beq
|\lambda_i-\lambda_i'| &\le &2\ep_1+\Big|\sum_{j=1}^{ l}\mu_jd_{i,j}-\sum_{j=1}^{l}\mu_jd'_{i,j}\Big|\\
&\le& 2\ep_1+\max_{1\le j\le {l}}\{|\mu_j|\}\Big(\sum_{i=1}^l|\ep_{i,j}|\Big)\\
&=&  2\ep_1+\max_{1\le j\le l}
\{|\mu_j|\}\Big(\sum_{i\in \Lambda_+}\ep_{i,j}+\sum_{i\in \Lambda_-}|\ep_{i,j}|\Big)\\
&\le & 2\ep_1+2\ep_2\|y\|.
\eneq
Consequently,
\beq
\|\phi_{D'}(y)-x\|\le \|\phi'(y)-x\|+\|\phi'(y)-\phi_{D'}(y)\|<3\ep_1+2 \ep_2\|y\|.
\eneq

In general case, there is a unitary $u\in A$ such that
$u^*q_iu=p_i,$ $i=1,2,...,l.$
Put $\phi'={\rm Ad}\, u\circ \phi$ and $x'=u^*xu.$
Then we can apply the above to obtain a trace preserving \cpc\, $\phi'_{D'}: A\to A$
such that
\beq
\|\phi'_{D'}(y)-x'\|<3\ep_1 +2\ep_2\|y\|.
\eneq
Let $\psi={\rm Ad}\, u^*\circ \phi'_{D'}.$  Then $\psi$ is trace preserving \cpc\, and
\beq
\|\psi(y)-x\|<3\ep_1+2\ep_2\|y\|.
\eneq

\end{proof}

\begin{lemma}\label{Ln04}
Let $A$ be a unital infinite dimensional  simple \CA\, with real rank zero and stable rank one and let
$p_1, p_2,...,p_{l}$ and $e$ be mutually orthogonal projections in $A.$
Let $x=x_1+x_2,$ where $x_1=\sum_{j=1}^{l} \lam_j p_j,$  and ${{x_2}}\in eAe$ be normal elements.
Suppose that $K>1$ is an integer such that
$(2K+5)[e]\le [p_j]$ ($1\le j\le l$)
and $\{\lam_1, \lam_2,...,\lam_l\}$ is $\eta$-dense in $\sp(x).$
Then, for any $\ep>0,$
$$
{\rm dist}(x, \conv(\uu(x_1))\le \frac{8\|x\|}{K} +\ep+\eta.
$$
\end{lemma}

\begin{proof}
Choose a projection $p_0\le p$ such that $[p_0]=[e].$
Note that, by 3.4 of \cite{FR}, since $A$ has stable rank one,
there exists a normal element $x_2'\in eAe$ such that  $\sp(x'_2)\subset \Gamma,$ where
$\Gamma$ is an one-dimensional  finite CW complex in the plane,
and $\|x_2'-x_2\|<\ep.$
Then, $\sp(x_2')=X_0\cup X_1,$ where $X_0\subset \C$ is a compact subset  with covering dimension
no more than 1 such that
$K_1(C(X_0))=\{0\}$ and where $X_1$ is an one-dimensional finite CW complex in the plane.
It follows from Lemma 3 of \cite{LinFU} that there exists $x_0\in {\cal N}(p_0Ap_0)$ with $\sp(x_0)=X_1$
such that,  for any $\lam\not\in\sp(x'_2),$
\beq
[\lam p_0-x_0]=-[\lam e-x'_2]\,\,\,{\rm in}\,\,\,K_1(A).
\eneq

Choose any mutually orthogonal projections $\{e_i:i=1,...,l\}\subset eAe$ such that $\sum_{i=1}^le_i=e.$

Choose $p_i'\le p_i$ such that $[p_i']=[e_i],$ $i=1,2,...,l.$
Define $x_3'=\sum_{j=1}^l \lam_jp_j'$ and $x_3=\sum_{j=1}^l \lam_j(p_j-p_j').$
We may write
\beq
x=x_3'+x_3+x_2.
\eneq
Note that, since $A$ has stable rank one,  one computes that
\beqq\label{neww5}
[p_j-p_j']= [p_j]-[p_j']\ge (2K+5)[e]-[e]=(2K+4)[e]\ge (2K+4)[p_j'],
\eneqq
$j=1,2,...,l.$
It follows from (3) of  \ref{Ln25}  and \ref{lem00} that
\beq
x\in_{4\|x\|/2K}  \conv(\uu(x_0+x_3+x_2)).
\eneq
Since $\lam (p_0+e)-(x_0+x_2)\in {\rm Inv}_0((p_0+e)A(p_0+e))$ for all $\lam\not\in \sp(x_2),$
by \cite{LinFU}, \wilog, we may assume
that
\beq
\|(x_0+x_2)-x_4\|<\ep+\eta,
\eneq
where $x_4=\sum_{j=1}^l \lam_j q_j,$ and    $\{q_1,q_2,...,q_l\}$ is a set of mutually orthogonal projections
in $(p_0+e)A(p_0+e).$

By \eqref{neww5},
$$(K+2)[q_i]\le(K+2)[p_0+e]=(2K+4)[e]\le [p_i-p_i'],i=1,...,l,$$
Thus, by part (3) of \ref{Ln25} (note $x_1=x_3'+x_3$),
\beq
x_4+x_3\in_{\frac{4\|x\|}{K}} \conv(\uu(x_1)).
\eneq
We conclude that
\beq
x\in_{\frac{8\|x\|}{K}+\ep+\eta} \conv(\uu(x_1)).
\eneq
\end{proof}

\begin{theorem}\label{alpha}
Let $A$ be  a unital  separable simple  $C^*$-algebra with tracial rank zero.
 Then  for any normal elements $x,y\in A$,  $x\in\overline{\conv(\uu(y))}$ if and only  if  there exists a sequence of  unital completely positive linear  maps $\psi_n: A\to A$ such that
\beqq\label{alpha-1}
&&\lim_{n\to\infty}\|\psi_n(y)-x\|=0\tand\\\label{alpha-2}
&&\lim_{n\to\infty}\sup\{|\tau(\psi_n(a))-\tau(a)|: \tau\in T(A)\}=0\tforal a\in A.
\eneqq
\end{theorem}

\begin{proof}
Let  $x,y
\in \nnn(A)$ and let $\{\psi_n\}: A\to A$ be a sequence
of unital \cpc s which satisfies \eqref{alpha-1} and \eqref{alpha-2}.

\Wlog,  we may assume that $\|x\|, \|y\|\le 1.$  By replacing $y$ by $y-\lambda$ for some $\lambda\in \sp(y)$ and $x$ by $x-\lambda,$
\wilog, we may assume that $0\in \sp(y).$

Fix $\ep>0.$
Put $\ep_0=\ep/2^{14}$ and an integer $K$ such that $0<12/K<\ep_0.$

It follows from  \ref{Csec3for4} that there exists a finite dimensional \SCA\, $B$ of $A$ with $p_0=1_B$ and there are $x_0,y_0\in \nnn({{B}})$, $x_0',y_0'\in\nnn((1-p_0)A(1-p_0))$ such
\beqq
x\approx_{\ep_0} x_0+x_0',\,\,\,\,
y\approx_{\ep_0} y_0+y_0',
\eneqq
where
$$
x_0=\sum_{i=1}^h\lam_ip_{i,0},\,\,\,
y_0=\sum_{k=1}^N{\mu_k}q_{k,0},
$$
where $\{\lam_1,...,\lam_h\}$ and $\{{\mu_1},...,{\mu_N}\}$ are  $\ep_0$-dense in $\sp(x)$ and $\sp(y)$ respectively,
${\mu_1}=0,$
and  $p_{i,0},q_{k,0}\in B$ are projections satisfying condition:
\beqq
&&\sum_{i=1}^hp_{i,0}=\sum_{k=1}^Nq_{k,0}=p_0\andeqn\\
&&\hspace{-0.5in}(2K+5)[1-p_0]\le [p_{i,0}],\,\,\, i=1,2,...,h
\andeqn\\
&&  (2K+5)[1-p_0]\le [q_{k,0}], k=1,2,...,N.
\eneqq
Put $q_{0,0}=1-p_0$ and
\beqq
s_0=\inf\{\tau(q_{k,0}): \tau\in T(A), \,\,\, k=0,1,...,N\}>0.
\eneqq

It follows from (2) of \ref{Ln25} that
there is a trace preserving unital \cpc\,  $\phi_1: A\to A$
such that
\beqq\label{F44-09-1}
\|\phi_1(y)-y_0\|<\ep_0+1/2K<2\ep_0  \andeqn y_0\in_{2\ep_0}\conv(\uu(y)).
\eneqq

It follows from \ref{Ln04} that
there is a second trace preserving unital \cpc\, $\phi_2: A\to A$
such that
\beqq\label{F44-09-2}
\|\phi_2(y_0)-y\|<\frac {8\|y\|}K+2\ep_0<3\ep_0,
\eneqq

Put $B_0=\C(1-p_0)\oplus B.$
Let  ${\cal P}_{B_0}$ be the set of all non-zero projections in $B_0.$
There are only finitely many unitary equivalence  classes of projections in ${\cal P}_{B_0}.$
Put
\beq
s_{00}=\inf\{\tau(e): \tau\in T(A)\}\andeqn e\in {\cal P}_{B_0}>0.
\eneq
Put
\beq
\ep_1=(s_{00}/64NK)\ep_0.
\eneq
Note that the unit ball of $B_0$ is compact.

Choose large $n_0$ so that
\beqq
\|\psi_{n_0}\circ \phi_2(y_0)-x\|<4\ep_0\andeqn
\eneqq
and for all $b\in B_0$,
\beqq\label{3m-n1}
\sup\{|\tau(\psi_{n_0}\circ \phi_2(b))-\tau(\phi_2(b))|:\tau\in T(A)\}\le (\ep_1/4)\|b\|.
\eneqq
Put $\phi_3=\psi_{n_0}\circ \phi_2.$ Since $\phi_2$ is trace preserving,  by \eqref{3m-n1},
\beqq
\sup\{|\tau(\phi_3(b))-\tau(b)|:\tau\in T(A)\}\le (\ep_1/4)\|b\|\rforal  b\in B_0.
\eneqq
We also have
\beqq\label{F44-09-3}
\|\phi_3(y_0)-x\|<4\ep_0,
\eneqq

There are
$x_{i,k}, y_{i,k}\in A$ such that
\beqq\label{trace-n1}
\Big\|\phi_3(q_{k,0})-\sum_{i=1}^{m(k)}x_{i,k}^*x_{i,k}\Big\|<\ep_1,\,\,\,
\Big\|\sum_{i=1}^{m(k)}x_{i,k}x_{i,k}^*-q_{k,0}\Big\|<\ep_1,\\\label{trace-n2}
\Big\|\sum_{i=1}^{m(k)'} y_{i,k}^*y_{i,k}-q_{k,0}\Big\|<\ep_1\andeqn
\sum_{i=1}^{m(k)'}y_{i,k}y_{i,k}^*\ge s_0\cdot 1_A,
\eneqq
$k=0,1,2,...,N.$
by

Note that we may assume, \wilog, (by splitting $x_{i,k}, y_{i,k}$ into  more terms), that  $\|x_{i,k}\|, \,\|y_{i,k}\|\le 1$ for all $i$ and $k.$

Put
\beq
{\cal F}_1=\{1_A, y_0,p_0,p_{1,0},...,p_{h,0};q_{0,0} q_{1,0},...,q_{N,0}\}\cup {\cal F}_0\cup\{\phi_3(y_0)\},
\eneq
where
$$
{\cal F}_0=\{x_{i,k}, x_{i,k}^*: 1\le i\le m(k), 0\le k\le N\}\cup \{y_{i,j}, y_{i,k}^*: 1\le i\le m(k)',0\le k\le N\}.
$$
Put $M_{00}=\max\{m(k), m(k)':  k=0,1,...,N\}.$

Let  $\ep_2=\ep_1/(2^{10}(M_{00}+1)NK).$
Choose a finite subset ${\cal F}_{B_0}$ of the unit ball of $B_0$
which is $\ep_2/4$-dense.
Since $B_0$ is projective, choose $\dt_0>0$ and a finite subset ${\cal G}_{B_0}$
such that, for  any unital ${\cal G}_{B_0}$-$\dt_0$-multiplicative \cpc\, $L'$ from $B_0$  to a
unital \CA\, $A',$ there exists a unital \hm\, $h': B_0\to A'$ such that
\beq
\|h'(b)-L'(b)\|<\ep_2/4\rforal b\in {\cal F}_{B_0}.
\eneq
Consequently,
\beq
\|h'(b)-L'(b)\|\le (\ep_2/2)\|b\|\rforal b\in B_0.
\eneq
We may assume that ${\cal G}_{B_0}$ contains a generating
set of $B_0.$

Put
$$
{\cal F}_2={\cal F}_1\cup {\cal G}_{B_0}\cup\{x, x'\}.
$$

Choose a large finite subset ${\cal F}_3$ of $A$ such that
for any non-zero element $z\in {\cal F}_2,$ there are $a_{z,i}, b_{z,i}\in {\cal F}_3$ such
that
\beqq\label{F44-n-100}
\sum_{i=1}^{r(z)} a_{z, i}zb_{z, i}=1_A.
\eneqq
Put $M_1=2\max\{\{r(z)\max\{\|a_{z, i}\|+\|b_{z,i}\|: 1\le i\le r(z)\}: z\in {\cal F}_2\setminus \{0\}\}.$
Let $\ep_3=\min\{\ep_2/M_1, \dt_0/M_1\}.$

By the virtue of \ref{Csec3for4}, there is   a finite dimensional \SCA\, $B_1$ with $1_{B_1}=p$ such
that
\beqq\label{F44-n-10-1}
&&\|pz-zp\|<\ep_3\rforal z\in {\cal F}_3,\\\label{F44-n-10}
&&{\rm dist}(pzp, B_1)<\ep_3\rforal z\in {\cal F}_3
\andeqn\\
&&\sup\{\tau(1-p): \tau\in T(A)\}<\ep_3.
\eneqq
as well as   ${\bar x}_1, y_1\in {\cal N}(B_1)$ ${\bar x}_2, y_2\in {\cal N}((1-p)A(1-p)),$
projections $q_{i,1}, p_{i,1}\in B_1$ and $q_{i,2}\in (1-p)A(1-p)$
such that
\beqq
&& pxp\approx_{\ep_3} {\bar x}_1,\,\, (1-p)x(1-p)\approx_{\ep_3} {\bar x}_2,\\\label{F44-n-10+0}
&& py_0p\approx_{\ep_3} y_1,\,\, (1-p)y_0(1-p)\approx_{\ep_3} y_2,\\\label{F44-n-10+1}
&& y_1=\sum_{i=1}^N\mu_iq_i,\,\,
{\bar x}_1=\sum_{j=1}^{N'} {{\lambda'_j}}p_j,\,\,
y_2=\sum_{i=1}^N\mu_i q_i'\\
&& q_{i0}\approx_{\ep_3} q_{i}+q_i',\,\,\,{{i=0,1,...,N}}\\
&&{{q_0+\sum_{i=1}^Nq_i=\sum_{i=1}^{N'}p_i=p,}}\\\label{F44-n-15+12}
&&\hspace{-0.3in}(2K+5)[1-p]\le [q_i]\andeqn (2K+5)[1-p]\le [p_j]\\
&&\rforal 0\le i\le N, \,\, 1\le j\le N',
\eneqq
where $\{{{q_0}}, q_1,q_2,...,q_N\}$  and
$\{p_1, p_2,...,p_{N'}\}$ and
$\{{{q_0}}, q_1',q_2',...,q_N'\}$ are mutually orthogonal projections   and
$\{\lam_1', \lam_2', ...,\lam_{N'}'\}$ is $\ep_3$-dense in $\sp(x).$
We may write
$B_1=M_{R(1)}(\C)\bigoplus M_{R(2)}\bigoplus \cdots \bigoplus M_{R(k_0)}$
and let $\pi_r: B_1\to M_{R(r)}(\C)$ be the quotient map, $r=1,2,...,k_0.$

Moreover, by the choice of $\dt_0,$ we may assume that there exists an isomorphism
$h_0: B_0\to B_1$ such that
\beqq\label{F44-n-9}
\|h_0(b)-pbp\|<\ep_2\|b\|\rforal b\in B_0
\eneqq
Moreover, by \eqref{F44-n-100} and the choice of  ${\cal G}_{B_0}$ and $\ep_3,$ we may assume
that $\pi_r\circ h_0$ is an isomorphism for each $r.$

Furthermore, we may assume that $y_1=h_0(y_0),$ $q_k=h_0(q_{k,0}),$
$k=1,2,...,N.$

Let $C_r=\pi_r(h_0(B_0)),$ $r=1,2,...,k_0.$
There is a conditional  expectation $E_r: M_{R(r)}\to C_r.$
Define $L^{(r)}: M_{R(r)}\to M_{R(r)}$ by
\beq
L^{(r)}=\pi_r\circ h_0\circ \phi_3\circ (\pi_r\circ h_0)^{-1}\circ E_r,\,\,\,
r=1,2,...,k_0.
\eneq  Define $L: B_1\to B_1$ by
$L(a_1\oplus a_2\oplus\cdots \oplus a_{k_0})=\bigoplus_{r=1}^{k_0}L^{(r)}(a_r)$
for $a_1\oplus a_2\oplus\cdots \oplus a_{k_0}\in B_1.$

Note that
\beqq\label{neww8}
&&\hspace{-0.3in}L(q_k)=
h_0\circ \phi_3(h_0^{-1}(q_k))=h_0\circ \phi_3(q_{k,0})\andeqn\\
&&\hspace{-0.3in}L(y_1)=
h_0\circ \phi_3(h_0^{-1}(y_1))=h_0\circ \phi_3(y_0).
\eneqq
We have
\beqq
\|L(y_1)-{\bar x}_1\| &\le& \|h_0(\phi_3(y_0))-p\phi_3(y_0)p\|+\|p\phi_3(y_0)p-pxp\|\\\label{F44-n-101}
&& +\|pxp-{\bar x}_1\|<\ep_2+4\ep_0+\ep_3.
\eneqq
By  \eqref{F44-n-15+12}, and applying  \ref{Ln04},
\beqq\label{aq1}
x\in\conv_{\eta_0}(\uu({\bar x}_1)),
\eneqq
where $\eta_0=\frac{8\|x\|}K+4\ep_3<\ep/4.$

For each $z\in {\cal F}_3,$ by \eqref{F44-n-10}, there is $L(z)\in B_1$ such that
\beqq\label{F44-n-20}
\|pzp-L(z)\|<\ep_3\andeqn \|pz^*p-L(z)^*\|<\ep_3.
\eneqq

{{
By \eqref{F44-n-9},\eqref{trace-n1}, \eqref{trace-n2},\eqref{F44-n-10-1} and \eqref{F44-n-20},
\beqq
&&L(q_k)=h_0\circ \phi_3(q_{k,0})\approx_{\ep_2} p\phi_3(q_{k,0})p
\approx_{\ep_1} \sum_{i=1}^{m(k)}px_{i,k}^*x_{i,k}p\\\label{F44-n-20000}
&&\approx_{m(k)\ep_3}\sum_{i=1}^{m(k)}px_{i,k}^*px_{i,k}p
\approx_{4m(k)\ep_3}\sum_{i=1}^{m(k)}L(x_{i,k})^*L(x_{i,k}).
\eneqq
Similarly
\beqq\label{F44-n-201}
&&q_{k}\approx_{\ep_3}pq_{k,0}p\approx_{\ep_1}
\sum_{i=1}^{m(k)}px_{i,k}x_{i,k}^*p
\approx_{5m(k)\ep_3}\sum_{i=1}^{m(k)}L(x_{i,k})L(x_{i,k})^*,
\\\label{F44-n-202}
&&pq_{k,0}p\approx_{\ep_1}\sum_{i=1}^{m(k)'}py^*_{i,k}y_{i,k}p
\approx_{5m(k)'\ep_3}\sum_{i=1}^{m(k)'}L(y_{i,k})^*L(y_{i,k})\andeqn\\\label{F44-n-203}
&&\sum_{i=1}^{m(k)'}L(y_{i,k})L(y_{i,k})^*\approx_{5m(k)'\ep_3}
\sum_{i=1}^{m(k)'}py_{j,k}y^*_{i,k}p\ge s_0p.
\eneqq
Note that
\beq
10m(k)\ep_3+10m(k)'\ep_3+\ep_2+3\ep_1<s_0/8.
\eneq
By \eqref{F44-n-20000}, \eqref{F44-n-201},\eqref{F44-n-202} and \eqref{F44-n-203},
we have
\beqq\label{F44-n-102}
&&t(L(q_k))\ge s_0/2\rforal t\in T(B_1),k=0,1,....,N,\\\label{F44-n-102+1}
&&|t(L(q_k))-t(q_k)|<(s_0/4)\ep_0,\rforal t\in T(B_1),k=0,1,....,N.
\eneqq}}

Define $\psi_r: M_{R(r)}(\C)\to  M_{R(r)}(\C)$ by $\psi_r=\pi_r\circ L|_{M_{R(r)}},$
$r=1,2,...,k_0.$
Put
\beq
&&y_{1,r}=\pi_r(y_1),\,\, x_{r}=\pi_r({\bar x}_1),r=1,2,...,k_0,\\
&&p_{i,r}=\pi_r(p_i),\,\,\,q_{k,r}=\pi_r(q_k),\,\,i=1,...,{ {N'}},\,\,k={{0}},1,...,N,\,r=1,2,...,k_0.
\eneq
Note that $y_{1,r}=\sum_{k=1}^N \mu_i q_{k,r},$ $r=1,2,...,k_0.$
For each $r,$  by \eqref{F44-n-101},
\beqq
\|L^{(r)}(y_{1,r})-x_r\|<5\ep_0+3\ep_3.
\eneqq
By \eqref{F44-n-102} and \eqref{F44-n-102+1},  for $t\in T(M_{R(r)}),$
\beqq
&&{{\max\{|t(L^{(r)}(q_{k,r}))-t(q_{k,r})|:0\le k\le N\}}}<{{(s_0/4)}}\ep_0\andeqn\\
&&\min\{t(q_{k,r}): 0\le k\le N\}\ge (s_0/2).
\eneqq
{{Let $Q_{1,r}=q_{0,r}+q_{1,r},Q_{k,r}=q_{k,r},k=2,...,N$, then}}
\beqq
&&{{\max\{|t(L^{(r)}(Q_{k,r}))-t(Q_{k,r})|:1\le k\le N\}<(s_0/2)\ep_0,}}\\
&&{{\min\{t(Q_{k,r}): 1\le k\le N\}\ge (s_0/2),}}
\eneqq
{{and  $y_{1,r}=\sum_{i=1}^N\mu_iQ_{k,r}$.}}
Applying \ref{Ln43}, there is, for each $r,$ a trace preserving \cpc\, $\Phi_r
: M_{R(r)}(\C)\to  M_{R(r)}(\C)$
such that
\beqq
&&\|\Phi_r(y_{1,r})-x_{r}\|<  2\ep_0\|y\|   +(5\ep_0+3\ep_3)<\ep/4\\
&&x_{r}\in_{\ep/4} \conv(\uu(y_{1,r})).
\eneqq

Applying this to each $r,$ we obtain a trace preserving \cpc\, $\Phi: B_1\to B_1$ such
that
\beqq\label{aq2}
\|\Phi(y_1)-{\bar x}_1\|<\ep/4\andeqn {\bar x}_1\in_{\ep/4}\conv(\uu(y_1)).
\eneqq

By part (2) of  \ref{Ln25} (recall that $\mu_1=0$),
\beqq\label{aq3}
y_1\in_{\eta_2}\conv(\uu(y_1+y_2)),\,\,{\rm where}\,\,\eta_2=\frac{\|y''_0\|}{K+1}<\ep_0.
\eneqq
By \eqref{F44-n-10-1},\eqref{F44-n-10+0}, \eqref{F44-n-10+1} and \eqref{aq3}
\beqq
y_1\in_{\eta_2+2\ep_3} \conv(\uu(y_0)).
\eneqq
It follows (by also \eqref{F44-09-1}) that
\beqq\label{F44-nf-1}
y_1\in_{\eta_2+2\ep_3+2\ep_0}\conv(\uu(y)).
\eneqq
It follows from \eqref{F44-nf-1}, \eqref{aq2} and \eqref{aq1} that
\beq
x\in_{\ep} \conv(\uu(y)).
\eneq
\end{proof}

\begin{corollary}\label{beta}
Let $A$ be  a unital separable simple  $C^*$-algebra with
tracial rank zero
and let $x,\, y\in A$ be two normal elements.
 Then   the following are equivalent:

 {\rm (1)} $x\in\overline{\conv(\uu(y))};$

 {\rm (2)} there exists a sequence of  unital completely positive maps $\psi_n: A\to A$ such that
\beq
&&\lim_{n\to\infty}\|\psi_n(y)-x\|=0{{\tand}}\\
&&\tau(\psi_n(a))=\tau(a)\rforal a\in A{{\tand\tforal \tau\in T(A);\tand}}
\eneq

 {\rm (3)}  there exists a sequence of  unital completely positive maps $\psi_n: A\to A$ such that
\beqq\label{alpha-1+}
&&\lim_{n\to\infty}\|\psi_n(y)-x\|=0\tand\\\label{beta-2+}
&&\lim_{n\to\infty}\sup\{|\tau(\psi_n(a))-\tau(a)|: \tau\in T(A)\}=0\tforal a\in A.
\eneqq
\end{corollary}

\begin{proof}

We have established that (1) and (3) are equivalent.
It is also clear that (2) implies (3), and (1) implies (2) follows from \ref{PPP}.

\end{proof}

\section{Normal elements  with small boundaries}

The following follows from \cite{Linalm}.

\begin{lemma}[\cite{Linalm}]\label{LLin}
Let $X$ be a compact subset of the plane,
 $\ep>0$ and let ${\cal F}\subset C(X)$ be a finite subset.
There is $\dt>0$ and a finite subset ${\cal G}\subset C(X)$ satisfying the following:
If $\phi: C(X)\to F$ is a ${\cal G}$-$\dt$-multiplicative \cpc, where $F$ is a finite dimensional \CA,
then there exists unital \hm\, $h: C(X)\to F$ such
that
\beq
\|\phi(f)-h(f)\|<\ep\rforal f\in {\cal F}.
\eneq

\end{lemma}

\begin{definition}\label{DSB}
Let $X$ be a compact metric space, let $A$ be a  \CA\, with $T(A)\not=\emptyset,$ and
let $T\subset T(A)$ be a subset.
Suppose that $\phi: C(X)\to A$ is a unital \hm.
We shall say $\phi$ has (SB)  property  with respect to $T,$ if, for any $\dt>0,$ there is a finite
open cover $\{O_1, O_2, ..., O_m\}$ of $X$ with $\max\{{\rm diam}(O_i):i=1,...,m\}<\dt$ such that
$\mu_\tau(\partial(O_j))=\mu_\tau({\overline{O}_j}\setminus O_j)=0,$ for all $\tau\in T,$ $i=1,2,...,m,$
where $\mu_\tau$ is the probability Borel measure induced
by the state $\tau\circ \phi.$

If $T=T(A),$ we shall simplify say that $\phi$ has (SB) property.

\end{definition}

The following is easily proved (see the proof of Lemma 3 of  \cite{Lin94}).

\begin{proposition}\label{LSBNep}
Let $A$ be a unital simple \CA\, with $T(A)\not=\emptyset.$
Suppose that $\phi: C(X)\to A$ is a unital \hm\, with (SB) property.
Then the following holds:
For any $\dt_1>0$ and $\eta>0,$ there exists $\dt_2>0$ with $\dt_2<\dt_1$
and there exists a
compact subset  $K$of  $X$ such that

{\rm (i)}  $X\setminus K$ is a finite disjoint union of
open subsets $\{O_j: 1\le j\le m\}$ with $\max\{{\rm diam}(O_j):1\le j\le m\}<\dt_1,$

{\rm (ii)} $Y_i'\cap Y_j'=\emptyset,$ if $i\not=j,$
where $Y_{j}'=\{x\in X: {\rm dist}(x, O_j)<\dt_2/16\},$ $j=1,2,...,m,$

{\rm (iii)} $\cup_{j=1}^m Y_j\supset X,$
where $Y_j=\{x\in X: {\rm dist}(x, O_j)<\dt_2\},$ $j=1,2,...,m,$
and
\beq
\hspace{-1.9in}{\rm (iv)} \hspace{0.12in} \mu_{{\tau}}(K)<\eta/4(m+1)\tforal \tau\in T(A).
\eneq
\end{proposition}

\begin{lemma}[see the proof of Lemma 2 of \cite{Lin94} and that of Lemma 4.8 of \cite{Lin96app}]\label{LSB}
Let $A$ be a unital simple \CA\, $T(A)\not=\emptyset$ such that
its extremal points $\partial_e(T(A))$ has countably many points.
Suppose that $X$ is a compact metric space and
$\phi: C(X)\to A$ is a unital \hm.
Then $\phi$ has (SB) property.
\end{lemma}

\begin{proof}
Let $\dt>0.$ For each $\xi\in X,$ consider $S_{\xi, r}=\{x\in X: {\rm dist}(x, \xi)=r\},$
where $0<r<\dt/2.$ Since $\partial _e(T(A))$ is countable,
there is $0<r_\xi<\dt/2$ such that
\beq
\mu_\tau(S_{\xi, r_\xi})=0\rforal \tau\in \partial_e(T(A)).
\eneq
It follows that
\beq
\mu_\tau(S_{\xi, r_\xi})=0\rforal \tau\in T(A).
\eneq
Let $O_\xi=\{x\in X: {\rm dist}(\xi, x)< r_\xi\}.$
Then $\cup_{\xi\in X} O_\xi=X.$
There are $\xi_1, \xi_2,...,\xi_m\in X$ such
that
\beq
\cup_{i=1}^m O_{\xi_i}\supset X.
\eneq
Note that $\mu_\tau(\partial(O_{{\xi_i}}))=0,$ $i=1,2,...,m.$
\end{proof}

\begin{lemma}\label{Lsbspliting}
Let $A$ be a unital  simple  \CA\, with tracial rank zero.
Suppose that $x\in A$ is a normal element with ${\rm sp}(x)=X$ and
$j_x: C(X)\to A$ is the induced monomorphism which has
(SB) property.  Then, for any $\ep>0,$ any $\sigma>0,$ and any finite subset
${\cal F}\subset C(X),$ there exists a finite subset ${\cal G}\subset C(X)$
and $\eta>0$ satisfying the following condition:
If $\phi: C(X)\to F$ is a unital \hm, where $F$ is a finite dimensional \SCA\,
of $A$ such that
\beqq\label{Lsb-n1}
\sup\{|\tau(\phi(g))-\tau(j_x(g))|: \tau\in T(A)\}<\eta\tforal g\in {\cal G},
\eneqq
 then there exists a unital \hm\, $\phi_1: C(X)\to F$ and a unital  \cpc\, $L:
F\to A$
such that $L\circ \phi_1(C(X))\subset j_x(C(X)),$
\beq
\|\phi(f)-\phi_1(f)\|<{{\ep}}\tforal f\in {\cal F},\\
\|L\circ \phi_{{1}}(f)-j_x(f)\|<\ep\,\,\,\tforal f\in{\cal F}\tand\\
\sup\{|\tau\circ L(a)-\tau(a)\}|: \tau\in T(A)\}\le \sigma\|a\|\tforal a\in F.
\eneq
\end{lemma}

\begin{proof}

Choose $\dt_1>0$ such that
\beqq\label{Lsb-n-1-1}
|f(x)-f(x')|<\min\{{{\ep/64}},\sigma/4\}\rforal f\in {\cal F},
\eneqq
if ${\rm dist}(x, x')<2\dt_1.$ Choose $\eta=\min\{\ep/64, \sigma/4\}.$
Suppose that $j_x: C(X)\to A$ is a  unital \hm\, with (SB) property.
\Wlog, we may assume that ${\cal F}$ is in the unit ball of $C(X).$

There is $\dt_2>0$ with $\dt_2<\dt_1/4$ and
there are compact subset $K\subset X,$ open subsets
$O_1, O_2,..., O_N$ and $Y_1, Y_2,...,Y_N$ of $X$ satisfy the condition (i), (ii),  (iii) and (iv) in {{ \ref{LSBNep}}}.
Let $K=X\setminus \cup_{i=1}^N O_j$ be as in {{\ref{LSBNep}}}.

Let $f_1, f_2,...,f_N$ be a partition of unity with compact support ${\rm supp}(f_j)\subset Y_j,$
$j=1,2,...,N.$
Let  $g_K\in C(X)$ be a function such that
$g_K(t)=0$ if ${\rm dist}(t, \cup_{j=1}^N O_j)\le \dt_2/64$ and $g_K(t)=1$ if ${\rm dist}(t, \cup_{j=1}^N O_j)\ge \dt_2/16$
and $0\le g_K(t)\le 1.$ Then ${\rm supp}(g_K)\subset K.$
Let ${\cal G}={\cal F}\cup \{f_j: 1\le j\le N\}\cup \{g_K\}.$
Put $\xi_j\in O_j,$ $j=1,2,...,N.$
\Wlog, we may assume that
\beqq\label{5Lfinited-n1}
\|\sum_{i=1}^N  f(\xi_j)f_j-f\|<\min\{\ep/16, \sigma/4\}\rforal f\in {\cal F}.
\eneqq
Define
\beqq
s_0=\inf\{\tau(j_x(f_j)): \tau\in T(A),\,1\le j\le N\}>0.
\eneqq
Let $\phi: C(X)\to F$ be a unital \hm\, for some finite dimensional \SCA\, $F$ of $A$
which satisfies \eqref{Lsb-n1} for  $\eta_0=\min\{\eta/8(N+1), s_0/2\}.$

Write $\phi(f)=\sum_{k=1}^n f(x_k) p_k$ for all $f\in C(X),$
where $p_1, p_2,...,p_n$ are mutually orthogonal projections and
$x_1, x_2,...,x_n$ are distinct points.  By the choice of $s_0$ and $\eta_0,$  and
by reordering, we may assume that $\{x_1,x_2,...,x_n\}\cap Y'_j\not=\emptyset,$
$j=1,2,...,N.$
Note that
\beqq\label{5Lfinited-11}
\phi(\sum_{i=1}^N f(\xi_i)f_i)=\sum_{i=1}^{N} f(\xi_i)\phi(f_i).
\eneqq
Note also that, by \eqref{Lsb-n1} and (iv) of {{\ref{LSBNep}}},
\beq
\tau(\phi(g_K))<\eta/2(N+1)\rforal \tau\in T(A).
\eneq
Put $K_1=\{x\in X: {\rm dist}(x, {{\cup_{j=1}^N O_j}})\ge \dt_2/16\}.$
It follows that
\beqq\label{5Lfinited-1}
\sum_{x_i\in K_1} \tau(p_i)<\eta/2(N+1).
\eneqq

Let $q_1=\sum_{x_j\in Y_1}p_j,$ $q_2=(1_F-q_1)\sum_{x_j\in Y_2}p_i,$...,
$q_{N}=(1-\sum_{i=1}^{N-1}q_i)(\sum_{x_j\in Y_{N}}p_i).$
Therefore $q_i$ is a sum of  all  projections $p_j$ with property that $x_j\in Y_i'$ and some
of projections $p_j$ with property $x_j\in Y_i\setminus Y_i'.$ Note that $q_i\not=0$ for all $i\in \{1,2,...,N\}$ and
$q_iq_j=0$ if $i\not=j$ and $\sum_{i=1}^{N} q_i=1_F.$
We may also write $q_i=\sum_{k\in S_i}p_k,$ where
$S_1\sqcup S_2\sqcup\cdots \sqcup S_N=\{1,2,...,m\}.$  Moreover, if
$j\in S_i,$ then $x_j\in {Y'_i}.$

Since $\sum_{j=1}^{N}f_j(x)=1$ for all $x\in X,$  then
 $f_j(x_i)=1$ if $x_i\in Y_j'.$  Also we have
\beq
\phi(f_j)=\sum_{x_i\in Y_j} f_j(x_i)p_i=(\sum_{x_i\in Y_j'}p_i)+(\sum_{x_i\in Y_j\cap K_1}f_j(x_i)p_i)
\eneq
It follows  from \eqref{Lsb-n1} and  \eqref{5Lfinited-1} that
\beqq\label{5Lfinited-2}
|\tau(q_j)-\tau(\phi(f_j))|\le \tau(g_K)<\eta/2(N+1)\rforal \tau\in T(A).
\eneqq
Let $C$ be the \SCA\, of $F$ generated by $q_1, q_2,...,q_{N}.$

Write $$F=M_{r(1)}\bigoplus M_{r(2)}\bigoplus\cdots \bigoplus M_{r(m)}$$ and
$\pi_j: F\to M_{r(j)}$ is the quotient map. Define $C_j=\pi_j(C),$   $j=1,2,...,m.$

Put $q_{i,j}=\pi_j(q_i)$ and $p_{k,j}=\pi_j(p_k)$  which we also view them as  projections in $F$ as well as projections in $A.$  

Put $g_i=f_i+(\eta/2N)\cdot 1,$ $i=1,2,...,N.$
Then, for fixed $i,$  by \eqref{5Lfinited-2}, $\sum_{j=1}^m\tau(q_{i,j})<\tau(j_x(g_i)),$ $i=1,2,...,N.$
Thus, by \ref{Lcp} we obtain $a_{i,j}\in A_+$ such that (again, viewing $q_{i,j}$ as projection in $A$),
\beqq
&&a_{i,j}\le j_x(g_i),\,\,\, \sum_{j=1}^m a_{i,j}=j_x(g_i)\andeqn \tau(a_{i,j})=\tau({{q_{i,j}}}),\,\,\,1\le j\le m-1,\\
&&\tau(a_{i,m})=\tau(j_x(g_i))-\sum_{j=1}^{m-1} \tau(a_{i,j})\andeqn\\\label{Lsb-12-2}
&&|\tau(a_{i,m})-\tau(q_{i,{{m}}})|\le \eta/2N,\,\,\,1\le i\le N
\eneqq
for all $\tau\in T(A).$

Define $\phi_1: C(X)\to C\subset F$ by
$\phi_1(f)=\sum_{i=1}^{N} f(\xi_i)q_i$ for all $f\in C(X).$
Then,  by  \eqref{Lsb-n-1-1},
\beqq\label{Lsb-12}
\|\phi_1(f)-\phi(f)\|<\min\{\ep/16, \sigma/4\}\rforal f\in {\cal F}.
\eneqq
Define $L_1: \bigoplus_{j=1}^m C_j\to A$ by
\beq
L_1(\sum_{i=1}^{N}(\sum_{j=1}^m\lambda_{i,j} q_{i,j}))=\sum_{i=1}^{N}(\sum_{j=1}^m\lambda_{i,j}{{a_{i,j}}})
\,\,\,{\rm for}\,\,\, \lambda_{i,j}\in \C.
\eneq
Then, by \eqref{5Lfinited-2} and \eqref{Lsb-12-2}, for all $c\in F,$
\beqq\label{Lsb-20}
|\tau\circ L_1(c)-\tau(c)|\le (\sigma/2)\|c\|\rforal \tau\in T(A).
\eneqq
We also note that
\beqq
L_1(\sum_{i=1}^{N}\lambda_i q_i)=\sum_{i=1}^{N}\lambda_i{{j_x}}(g_i)\in j_x(C(X)).
\eneqq
In other words, $L_1$ maps $C$ into $j_x(C(X)).$
Moreover,  by \eqref{5Lfinited-n1},
\beqq\label{Lsb-13}
&&\hspace{-0.4in}\|L_1\circ \phi_1(f)-j_x(f)\|=\|L_1(\sum_{i=1}^{N}f(\xi_i)q_i)-j_x(f)\|\\\label{Lsb-13+}
&&\le \|\sum_{i=1}^{N} f(\xi_i)j_x(g_i)-\sum_{i=1}^N f(\xi_i){{j_x(f_i)}}\|+\|\sum_{i=1}^N f(\xi_i){{j_x(f_i)}}-j_x(f)\|\\
&&<\sum_{i=1}^N| f(\xi_i)|(\eta/2N)+\ep/16<\ep/4\rforal f\in {\cal F}.
\eneqq

Since $M_{r(j)}$ and $C_j$
 are von-Neumann algebras, there exists a conditional expectation
$E_j: M_{r(j)}\to C_j$ such that
\beq
t(E_j(a))=t(a)\rforal a\in M_{r(j)},\,\,{\rm where}\,\,t\in T(M_{r(j)}),\,\,\, j=1,2,...,m.
\eneq
Consequently,
\beqq\label{Lsb-14-1}
\tau(E_j(a))=\tau(a)\rforal a\in M_{r(j)} \andeqn \tau\in T(A).
\eneqq
Define $E: F\to \bigoplus_{i=1}^m C_i$ by $E=\oplus_{i=1}^m E_i\circ \pi_i.$
Note $E^2=E.$
Hence we also have, by \eqref{Lsb-12},
\beqq\label{Lsb-14}
\|E\circ \phi(f)-\phi_1(f)\|=\|E\circ\phi(f)-E\circ \phi_1(f)\|<\ep/4\tforal f\in {\cal F}.
\eneqq
Put $L=L_1\circ E: F\to A.$  Then
$L_1(\phi_1(C(X)))=j_x(C(X).$ Moreover, for any $f\in C(X),$  by \eqref{Lsb-14} and \eqref{Lsb-13+},
\beq
&&\|L\circ \phi(f)-j_x(f)\|\\
&\le& \|L_1\circ E\circ \phi(f)-L_1\circ E\circ \phi_1(f)\|+\|L_1\circ E\circ \phi_1(f)-j_x(f)\|\\
&\le& \ep/4+\|L_1\circ \phi_1(f)-j_x(f)\|<\ep/4+\ep/4\rforal f\in {\cal F}.
\eneq
Furthermore, by \eqref{Lsb-14-1} and \eqref{Lsb-20},
\beq
&&|\tau\circ L(a)-\tau(a)|=|\tau\circ L(a)-\tau(E(a))|
=|\tau(L_1(E(a))-\tau(E(a))|\\
&\le& (\sigma/2)\|E(a)\|
\le(\sigma/2)\|a\|\rforal a\in F.
\eneq

\end{proof}

\begin{lemma}\label{2Lsbspliting}
Let $A$ be a unital  simple  \CA\, with tracial rank zero and with a unique
tracial state.
Suppose that $x\in A$ is a normal element with ${\rm sp}(x)=X$ and
$j_x: C(X)\to A$ is the induced monomorphism.
Then, for any $\ep>0,$ any $\sigma>0,$ and any finite subset
${\cal F}\subset C(X),$ there exists a finite subset ${\cal G}\subset C(X)$
and $\eta>0$ satisfying the following condition:
If $\phi: C(X)\to F$ is a unital \hm, where $F$ is a finite dimensional \SCA\,
of $A$ such that
\beqq\label{Lsb-n11}
{{|\tau(\phi(g))-\tau(j_x(g))|}}<\eta\tforal g\in {\cal G},
\eneqq
 then there exists
 a unital  \cpc\, $L:
F\to j_x(C(X))$
such that
\beq
&&\|L\circ \phi(f)-j_x(f)\|<\ep\,\,\,\tforal f\in{\cal F}\tand\\
&&|\tau\circ L(a)-\tau(a)|\le \sigma\|a\|\tforal a\in F.
\eneq
\end{lemma}

\begin{proof}
By the assumption and \ref{LSB},  $j_x$ has (SB) property.
The proof is a simplification of that of  \ref{Lsbspliting}.
We keep all lines of the proof of \ref{Lsbspliting} until
$C$ is defined.  Let $\tau$ be the only tracial state of $A.$
Then since both $F$ and $C$ are von Neumann algebras, there
is a conditional expectation $E: F\to C$ such that
\beqq
\tau(E(a))=\tau(a)\rforal a\in F.
\eneqq
Define $L_1:C\to j_x(C(X))$ by
$L_1(q_i)={{j_x(f_i)}},\,\,\, i=1,2,...,N.$
This implies, by  \eqref{Lsb-20},
that
\beqq\label{C52-n-1}
|\tau(L_1(c))-\tau(c)|\le (\sigma/4) \|c\|\rforal c\in C.
\eneqq
Let $\phi_1$ be as the same as defined in the proof of \ref{Lsbspliting}.
Note that $\phi_1(C(X))=C$ since $q_i\in C,$ $i=1,2,...,N.$
Define $L=L_1\circ E.$
Then
\beqq
&&\|L(\phi_1(f))-j_x(f)\|=\|L(\sum_{i=1}^N f(\xi_i)q_i-j_x(f)\|\\
&&= \|\sum_{i=1}^N f(\xi_i)j_x(f_i)-j_x(f)\|<\min\{\ep/16, \sigma/4\}\rforal f\in {\cal F}.
\eneqq
It follows that, as $\|\phi_1(f)-\phi(f)\|<\min\{\ep/16, \sigma/4\}$  for all $f\in {\cal F}$ (by \eqref{Lsb-12}),
\beqq
&&\|L(\phi(f))-j_x(f)\|\le   \|L(\phi(f))-L(\phi_1(f))\|\\
&&\hspace{0.4in}+\|L(\phi_1(f))-j_x(f)\|
<\min\{\ep/8, \sigma/2\}\rforal f\in {\cal F}.
\eneqq
We also have, by \eqref{C52-n-1},
\beq
|\tau\circ L(a)-\tau(a)| &=&|\tau\circ L_1\circ E(a)-\tau(a))|\\
&\le & |\tau(L_1(E(a))-\tau(E(a))|\ +|\tau(E(a))-\tau(a)|\\
&\le& (\sigma/2)\|E(a)\|
\le  (\sigma/2)\|a\|\rforal a\in F.
\eneq
\end{proof}

\begin{theorem}\label{51}
Let $A$ be a unital separable {{simple}}
 \CA\, with tracial rank zero.
Suppose that $x, y\in A$ are two normal elements with ${\rm sp}(x)=X$ and ${\rm sp}(y)=Y.$
Denote by $j_x: C(X)\to A$ and $j_y: C(Y)\to A$
the embedding given by $j_x(f)=f(x)$ for all $f\in C(X)$ and $j_y(g)=g(y)$ for all $g\in C(Y).$
Suppose that $j_x$ has (SB) property and
suppose that there exists a sequence of  unital  positive linear maps $\Phi_n: C(X)\to C(Y)$ such that
\beqq
&&\hspace{-0.3in}\lim_{n\to \infty}\|\Phi_n(z_x)-y\|=0\tand\\
&&\hspace{-0.3in}\lim_{n\to\infty}\sup\{|\tau(\Phi_n(f)(y))-\tau(f(x))|: \tau\in T(A)\}=0\tforal f\in C(X),
\eneqq
where $z_x$ is the identity function on $\sp(x)=X.$
Then $y\in \overline{{\rm conv}({\cal U}(x))}.$
\end{theorem}

\begin{proof}

\Wlog, we may assume that $\|x\|, \|y\|\le 1.$
To show $y\in \overline{\conv(\uu(x))},$ \wilog, we may assume that $0\in {\rm sp}(x),$ as in the beginning of the proof of
 Theorem \ref{alpha}.
Let $z_x\in C(X)$ be the identity function on $X$ so that $j_x(z_x)=x.$
Let $\ep>0$
and $\sigma>0.$
Let ${\cal F}=\{1, z_x\}\subset C(X)$ be a finite subset.



Choose $\dt_0>0$ such that
\beqq\label{46-n-1}
|f(t)-f(t')|<\ep/64\rforal f\in {\cal F},\,\,\,{if}\,\,\, |t-t'|<\dt_0.
\eneqq
Choose $\dt_{00}=\min\{\dt_0, \ep/2^{10}\}.$
One may write
$$
X=\cup_{i=1}^N {{{\bar O}}}_i,
$$
where ${{O_i}}$ is an open subset of $X$ with diameter no  more than $\dt_{00}/3$ and
${{O_i\cap O_j}}=\emptyset$ if $i\not=j.$
{ {Choose $\xi_i\in O_i$ and $d_i>0$ with  $$O(\xi_i, d_i)=\{x\in X: {\rm dist}(\xi_i, x)<d_i\}\subset  O_i,i=1,2,...,m.$$}}
 Since $0\in X,$ as we assumed,
\wilog, we may also assume that $\xi_1=0.$

Since  $A$ is  a unital simple \CA, for each $\xi\in X,$
\beqq
\inf\{\mu_{\tau}(O(\xi, d_i/2): \tau\in T(A)\}>0,
\eneqq
 where $\mu_\tau$ is the Borel probability measure induced by
$\tau\circ j_x.$
Then
\beqq
s_0=\inf \{\inf\{\mu_\tau(O(\xi_i, {{d_i/2}})):\tau\in T(A)\}: 1\le i\le N\}>0.
\eneqq

For each $i\in \{1,2,...,N\},$ choose  $g_i\in C(X)$ with $0\le g_i\le 1$  such that
$g_i(t)=1$ if $t\in O({{\xi_i}}, {{d_i}}/2)$ and $g_i(t)=0$ if $t\not\in O_i,$ $i=1,2,...,N.$
Note that $g_ig_j=0$ if $i\not=j.$
Put ${\cal F}_1={\cal F}\cup \{g_i: 1\le i\le N\}.$
Note $\{\xi_1, \xi_2,...,\xi_N\}$ is $\dt_{00}$-dense in $X.$


Choose  an integer $K>0$ such that $12 \cdot 2^8/K<\ep.$
Let $$\ep_1=\min\{\ep/2^{10}(N+1), \sigma/8(K+1)(N+1), s_0\ep/2^{10}(K+1)(N+1)\}.$$

Let $\eta>0$ and finite subset set ${\cal G}\subset C(X)$  be
given by  \ref{Lsbspliting} for $\ep_1/4$ (in place of $\ep$) and $\ep_1/2$ (in place of $\sigma$).
We may assume that ${\cal F}_1\subset {\cal G}.$
\Wlog, we may also assume that $\|g\|\le 1,$ if $g\in {\cal G}.$

Choose a finite subset ${\cal G}_X\subset C(X)$ (in place of ${\cal G}$)  and $\dt_1>0$ (in place of $\dt$) be given by  \ref{LLin} for ${\cal G}$ (in place of ${\cal F}_1$) and $\ep_2/2$ (in place of $\ep$) as well as $X.$
Put $\ep_3=\min\{\ep_2/4, \dt_1/2\}.$

Fix a finite subset ${\cal F}_A\subset A.$
Let us assume that $\|a\|\le 1$ if $a\in {\cal F}_A$ and
$$
{\cal F}_A\supset j_x({\cal G})\cup j_x({\cal G}_X)\cup\{y\}.
$$

Since $A$ has tracial rank zero,  there is a finite dimensional \SCA\, $F_1\subset A$ with
$1_F=p$ and an  ${\cal F}_A$-$\ep_2$-multiplicative \cpc\, $\psi: A\to F_1$ such that
\beqq
\|ap-pa\|<\ep_3\rforal a\in {\cal F}_A,\\\label{51-m-9}
\|a-((1-p)a(1-p)\oplus \psi(a))\|<\ep_3\rforal a\in {\cal F}_A,\\\label{51-m-10}
\tau(1-p)<\ep_3/16 \rforal \tau\in T(A).
\eneqq
By applying  \ref{LLin},
 there is a unital \hm\, $\phi': C(X)\to F$ such that
\beqq\label{51-m-10-1}
\|\phi'(g)-\psi(g)\|<\ep_2/2\rforal g\in {\cal G}.
\eneqq
Moreover,
\beqq
&&\sup\{|\tau \circ \phi'(g)-\tau(j_x(g))|: \tau\in T(A)\}<\ep_2/2\rforal g\in {\cal G}\andeqn\\\label{51-m-9++}
&&\tau(\phi'(g_i))>63s_0/64\rforal \tau\in T(A),\,\,i=1,2,...,N.
\eneqq

It follows from  \ref{Lsbspliting} that there exists a unital \cpc\, $L_1: F_1\to A$  and a  unital \hm\,
$\phi_1: C(X)\to F_1$ such that
\beqq\label{51-m-1000}
&&L_1(\phi_1(C(X)))\subset j_x(C(X)),\\\label{51-m-10+}
&&\|\phi_1(f)-\phi'(f)\|<\ep_1/4\andeqn\\\label{51-m-11}
&&\|L_1\circ \phi'(f)-j_x(f)\|<\ep_1/4\rforal f\in {\cal G}, \andeqn\\\label{51-m-12}
&&|\tau\circ L_1(c)-\tau(c)|\le (\ep_1/2)\|c\|\rforal \tau\in T(A)\andeqn c\in F_1.
\eneqq
By \eqref{51-m-9++} and \eqref{51-m-10+}, we  have
\beqq
\tau(\phi_1(g_i))\ge 15s_0/16\rforal \tau\in T(A).
\eneqq

Write $\phi_1(f)=\sum_{i=1}^{m'} f(t_i)p_i$ for all $f\in C(X),$   where $t_i\in X$ and
$p_1, p_2,...,p_{m'}$ are mutually orthogonal projections in $F.$
We may also write
\beqq
\phi_1(f)=\sum_{k=1}^N (\sum_{t_i\in O_k}f(t_i)p_i)+(\sum_{t_i\in X\setminus \cup_{k=1}^N O_k} f(t_i)p_i)
\eneqq
for all $f\in C(X).$
Note that
\beqq\label{51-m-14}
\sum_{t_i\in O_k}p_i\ge \phi_1(g_k),\,\,\,k=1,2,...,N.
\eneqq
Define $q_1=\sum_{t_i\in {\bar O}_1} p_i,$ $q_2=(1-q_1)(\sum_{t_i\in {\bar O}_2}p_i),
..., q_N=(1-\sum_{i=1}^{N-1}q_i)(\sum_{t_i\in {\bar O}_N}p_i).$
Note that
\beqq
q_k\ge \sum_{t_i\in O_k}p_i\ge \phi_1(g_k),\,\,\,k=1,2,...,N.
\eneqq
Define  \hm\, $\phi_2: C(X)\to F$ by
$\phi_2(f)=\sum_{i=1}^N f(\xi_i)q_i\rforal f\in C(X).$
We have (using \eqref{46-n-1})
\beqq\label{51-m-15-}
&&\|\phi_2(f)-\phi_1(f)\|<\ep/64\rforal f\in {\cal G}\andeqn\\\label{51-m-15}
&&\tau(q_i)>15s_0/16\rforal \tau\in T(A),\,\,\,i=1,2,...,N.
\eneqq
Note
that $\phi_2(C(X))\subset \phi_1(C(X)).$  Therefore $L_1(\phi_2(C(X)))\subset j_x(C(X)),$
in particular,
\beqq\label{51-m-18}
q_k\in \phi_2(C(X))\andeqn L_1(q_k)\in j_x(C(X)),\,\,\,k=1,2,...,N
\eneqq
By  \eqref{51-m-9}, \eqref{51-m-10-1},  \eqref{51-m-10+},  and  \eqref{51-m-15-},
since $\xi_1=0$ and \eqref{51-m-15},  applying {{ (2)}} of Lemma \ref{Ln25},
we have
\beqq\label{51-m-19}
\phi_2(z_x)\in_{4/(K+1)} \conv(\uu(x)).
\eneqq

Suppose that $\Phi: C(X)\to C(Y)$ is a unital positive linear map such that
\beqq\label{51-m-2}
&&\|\Phi(z_x)-y\|<\ep_4/4\andeqn\\\label{51-m-3}
&&\hspace{-0.4in}\sup\{|\tau\circ \Phi(g)-\tau(g)|:\tau\in T(A)\}<\ep_4/4
\rforal \tau\in T(A)
\eneqq
and for all
$g\in L_1(\psi({\cal F}_A))\cup \{L_1(q_i): 1\le i\le N\} .$
Note that, by \eqref{51-m-3} and \eqref{51-m-12},
\beqq
|\tau(\Phi(L_1(q_i)))-\tau(q_i)| &\le& |\tau(\Phi(L_1(q_i)))-\tau(L_1(q_i))|\\
&&\hspace{0.2in}+|\tau(L_1(q_i))-\tau(q_i)|\\\label{51-m-42}
&<&\ep_4/4+\ep_1/2<\ep_1.
\eneqq
Similarly, by \eqref{51-m-15},
\beqq\label{51-m-43}
\tau(\Phi(L_1(q_i)))>3s_0/4\rforal \tau\in T(A).
\eneqq

Put ${\cal F}_Y=L_1({\cal G})\cup \{L_1(q_i): 1\le i\le N\}\cup \{1, z_y\},$
where $z_y\in C(Y)$ is the identify function on $Y$ so that $j_y(z_y)=y.$
Let ${\cal G}_Y\subset C(Y)$ and $\dt_Y>0$ be given by \ref{LLin}
for ${\cal F}_Y$ (in place of ${\cal F}$) and $\ep_3$ (in place of $\ep$) as well as  $Y$ (in place of $X$).

Let $\ep_4=\min\{\ep_3/2, \dt_Y/2\}$ and $F_0=\C(1-p)\oplus F_1.$

Choose a large finite subset ${\cal F}_A'\subset A$ such that
${\cal F}_Y\cup {\cal G}_Y\subset {\cal F}_A'.$ We may also assume that
${\cal F}_A'$ contains an $\ep_3$-dense subset of the unit ball of $F_1.$
By choosing even smaller $\ep_4$ and larger ${\cal F}_A',$  since $F_1$ is projective, we may assume
that there is a \hm\, $h'$ from $F_1$ such that
\beqq\label{51-m-30}
\|h'(a)-L'(a)\|<\ep_3\|a\| \rforal a\in F_1
\eneqq
 for any ${\cal F}_A'$-$\ep_4$-multiplicative \cpc\, $L'$  from $F_1$ to any \CA.

Since $A$ has tracial rank zero,   by applying  \ref{Csec3for4}, there are finite dimensional \SCA\, $F_2$ with
$e=1_{F_2}$ and a  ${\cal F}_A'$-$\ep_4$-multiplicative \cpc\, $\psi_1: A\to F_2$ such that
\beqq
&&\|ae-ea\|<\ep_4\rforal a\in {\cal F}_A',\\\label{51-m-20}
&&\|a-((1-e)a(1-e)\oplus \psi_1(a))\|<\ep_4\rforal a\in {\cal F}_A',\\\label{51-m-21}
&&y\approx_{\ep_4} y_0+y_1,\,\,\, y_0=\sum_{i=1}^l \lambda_i e_i\,\,{\rm and}\,\, y_1\in {\cal N}((1-e)A(1-e)),\\\label{51-m-21+1}
&&\tau(1-e)<\ep_4/16 \rforal \tau\in T(A)\andeqn\\\label{51-m-21+2}
&&(2K+5)[1-e]\le [e_i],\,\,\,i=1,2,...,l.
\eneqq
We also assume that $\{\lambda_1, \lambda_2,...,\lambda_l\}$ is
$\ep_4$-dense in $Y.$
By \eqref{51-m-42} and \eqref{51-m-43},  as in the proof of  \ref{alpha}, by choosing sufficiently large ${\cal F}_A',$ we may assume
that
\beqq\label{51-m-25}
&&|t(\psi_1(\Phi\circ L_1(q_k)))-t(\psi_1(q_k))|<2\ep_1\rforal t\in T(F_2)\andeqn\\\label{51-m-25+}
&&t(\psi_1(q_k))\ge s_0/2\rforal t\in T(F_2),\,\,\, 1\le k\le N.
\eneqq
Moreover,
\beqq\label{51-m-26}
t(\psi_1(1-e))< \ep_3\rforal t\in T(F_2).
\eneqq
We may further assume that $\psi_1(b)\not=0$ for $b$ in an $\ep_3$-dense subset of the unit ball of $F_1.$

By  \ref{LLin}, there is a unital \hm\, $h_Y: C(Y)\to F_2$ such that
\beqq\label{51-m-27}
\|h_Y(b)-\psi_1(b)\|<\ep_3\rforal  b\in  {\cal F}_Y
\eneqq
as well as a normal element $x'\in (1-e)A(1-e)$ such that
\beqq\label{51-m-27++}
\phi_2(z_x)\approx_{2\ep_3} x'+h_Y(\phi_2(z_x)).
\eneqq
Note, that
\beqq\label{51-m-27+}
\|h_Y(z_y)-y_0\|<\ep_4+\ep_3+\ep_4<3\ep_3/2.
\eneqq
Moreover, there is  a \hm\, $h_F: F_1\to F_2$ such that
\beqq\label{51-m-28}
\|h_F(a)-\psi_1(a)\|\le \ep_3\|a\|\rforal a\in F_1.
\eneqq
Since $\psi_1(b)\not=0$ for those $b$ in an $\ep_3$-dense subset of the unit ball of $F_1,$
$h_F$ is injective. Define $L_F: h_F(\phi_2(C(X)))\to F_2$ by
$L_F=h_Y\circ \Phi\circ L_1\circ (h_F|_{h_F(\phi_2(C(X)))})^{-1}.$
$L_F$ is a \cpc. One can extends $L_F$ to be defined on $h_F(F_1).$ It can further extended
to a \cpc\, on $F_2.$
We note
that
\beqq\label{51-m-29}
L_F(h_F(q_k))=h_Y\circ \Phi\circ L_1(q_k)\andeqn L_F(h_F(\phi_2(z_x))=h_Y\circ \Phi\circ L_1(\phi_2(z_x)).
\eneqq
It follows from 
\eqref{51-m-27}, \eqref{51-m-25}   and \eqref{51-m-28} that
\beqq
&&\hspace{-0.3in}|t(L_F(h_F(q_k)))-t(h_F(q_k))|\le |t(h_Y(\Phi\circ L_1(q_k)))-t(\psi_1(\Phi\circ L_1(q_k)))|\\
&&+|t(\psi_1(\Phi_1\circ L_1(q_k)))-t(\psi_1(q_k))|+|t(\psi_1(q_k))-t(h_F(q_k))|\\\label{51-m-60}
&&<\ep_3+2\ep_1+\ep_3<3\ep_1<(s_0/4)(\ep/16),\,\,\, k=1,2,...,N.
\eneqq
Put ${\bar q}_0=(e-\sum_{i=2}^Nh_F(q_i))+h_F(q_1).$
Then, by \eqref{51-m-26} and by \eqref{51-m-28},
\beqq\label{51-m-63}
|t(L_F({\bar q}_0))-t({\bar q}_0)|<3\ep_1+2\ep_3<(s_0/4)(\ep/16)
\rforal t\in T(F_3).
\eneqq
By \eqref{51-m-28} and \eqref{51-m-25+},
\beqq\label{51-m-64}
t(h_F(q_k))\ge s_0/2-\ep_3\ge s_0/4\rforal t\in T(F_2).
\eneqq
Also
\beqq\label{51-m-65}
t({\bar q}_0)\ge t(q_1)\ge s_0/4 \rforal t\in T(F_3).
\eneqq
By \eqref{51-m-29}, \eqref{51-m-10+}, \eqref{51-m-11},  \eqref{51-m-15-},  and by \eqref{51-m-2}
\beqq
L_F(h_F\circ \phi_2(z_x))\approx_{\ep/32} h_Y\circ \Phi(j_x(z_x))\approx_{\ep_3/4} h_Y(y)\approx_{3\ep_3/4} y_0.
\eneqq
Let $\ep_5=\ep/32+\ep_3/4+3\ep_3/4.$
Note that $h_F\circ \phi_2(z_x)=\sum_{i=2}^N\lambda_i h_F(q_i)+0\cdot {\bar q}_0$ and $L_F$ now is defined
on $F_2.$
By applying {\ref{Ln43}}, we  obtain
a trace preserving unital \cpc\, $\Psi_F: F_2\to F_2$ such that
\beqq\label{51-m-71}
y_0\in_{2(\ep/16)+3\ep_5} \conv(\uu(h_F\circ \phi_2(z_x))).
\eneqq
By \eqref{51-m-21+2} and \ref{Ln04},
\beqq\label{51-m-72}
y\in_{(12/K)+\ep_4+\ep_4} \conv(\uu(y_0)).
\eneqq
By
the fact $\xi_1=0,$ and
by \eqref{51-m-64}, {{(2) of   \ref{Ln25}}}, and by \eqref{51-m-27++}
\beqq\label{51-m-73}
h_F(\phi_2(z_x))\in_{4/(K+1)+2\ep_3} \conv(\uu(\phi_2(z_x))
\eneqq
By \eqref{51-m-19}, \eqref{51-m-73}, \eqref{51-m-71} and \eqref{51-m-72}, we obtain
\beq
y\in_{\ep} \conv(\uu(x)).
\eneq

\end{proof}

\begin{corollary}\label{52cc}
Let $A$ be a unital separable {{simple}}  \CA\, with tracial rank zero and  with countably many
extremal tracial states.
Suppose that $x, y\in A$ are two normal elements with ${\rm sp}(x)=X$ and ${\rm sp}(y)=Y.$
Denote by $j_x: C(X)\to A$ and $j_y: C(Y)\to A$
the embedding given by $j_x(f)=f(x)$ for all $f\in C(X)$ and $j_y(g)=g(y)$ for all $g\in C(Y).$
Suppose
that there exists a sequence of  unital  positive linear maps $\Phi_n: C(X)\to C(Y)$ such that
\beqq
&&\lim_{n\to \infty}\|\Phi_n({{z_x}})-{{y}}\|=0\tand\\
&&\hspace{-0.2in}\lim_{n\to\infty}\sup\{|\tau(\Phi_n(f)(y))-\tau(f(x))|: \tau\in T(A)\}=0\tforal f\in C(X).
\eneqq
Then $y\in \overline{{\rm conv}({\cal U}(x))}.$
\end{corollary}

{{
\begin{theorem}\label{Trr0}
Let $A$ be a unital separable  simple  \CA\, with real rank zero, stable rank one and
weakly unperforated $K_0(A),$ and let $x, y\in A$ be two normal elements.
Suppose either the embedding $j_x$ has (SB) property or $T(A)$ has countably many extremal points,
and suppose that
exists a sequence of  unital  positive linear maps $\Phi_n: C(X)\to C(Y)$ such that
\beqq
&&\lim_{n\to \infty}\|\Phi_n({{x}})-{{y}}\|=0\tand\\
&&\hspace{-0.2in}\lim_{n\to\infty}\sup\{|\tau(\Phi_n(f)(y))-\tau(f(x))|: \tau\in T(A)\}=0\tforal f\in C(X).
\eneqq
Then $y\in \overline{{\rm conv}({\cal U}(x))}.$
\end{theorem}
}}

\begin{proof}
{{It follows from  Theorem 4.5 of \cite{LinKK} that
there exists a unital simple AH-algebra $B$ with real rank zero and no dimension growth
such that there is a unital monomorphism $H: B\to A$ which induces the following identification:
\beq
(K_0(B), K_0(B)_+, [1_B], K_1(B))=(K_0(A), K_0(A)_+, [1_A], K_1(A)).
\eneq
Since both $A$ and $B$ have real rank zero,
by \cite{BH}, $\rho_A(K_0(A))$ is dense in $\Aff(QT(A))$ and
$\rho_B(K_0(B))$ is dense in $\Aff(T(B)).$ It follows that $H$ induces
an affine isomorphism $H_{\sharp}$ from $\Aff(T(B))$ onto  $\Aff(QT(A)).$}}

{{Fix $x, y\in {\cal N}(A).$ Let $X={\rm sp}(X)$ and let $j_x: C(X)\to A$ be the embedding
given by $j_x(f)=f(x)$ for all $f\in C(X),$ and let $\gamma: C(X)_{s.a.}\to \Aff(QT(A))$
be given by $\gamma(f)(\tau)=\tau(j_x(f))$ for all $\tau\in QT(A).$
Note that $[j_x]$ and $\gamma$ are compatible. Note also
that $\tau\circ j_x$ is a tracial state on $C(X).$
It follows from 5.3 of \cite{Linrange} that there is a normal element
$x_1\in B$ with $\sp(x_1)=X$ and $[j_{x_1}]=[j_x]$ and
$\tau(j_{x_1}(f))=H_{\sharp}^{-1}\circ \gamma(f)(\tau)$ for all $\tau\in QT(A),$
where $j_{x_1}: C(X)\to B\subset A$ is induced by $x_1.$
Then, by Theorem 5.6 of \cite{HL} ($T(A)$ there should be $QT(A)$),
$x_1$ and $x$ are approximately unitarily equivalent.}}

{{Exactly the same argument shows that there is $y_1\in {\cal N}(B)$ such that $y_1$ and $y$ are approximately
unitarily equivalent. and  there exists a unital injective \hm\,
$j_{y_1}: C(Y)\to B\subset A$ is induced by $y_1,$
}}
Note that, by \cite{Linlondon}, $B$ has tracial rank zero.
{{By \ref{51} or \ref{52cc},
 $x_1\in \overline{\conv({\cal U}(y_1))}.$
 It follows that $x_1\in \overline{\conv({\cal U}(y))},$ whence $x\in \overline{\conv({\cal U}(y))}.$}}
\end{proof}

\begin{theorem}\label{52-m}
Let $A$ be a unital separable simple \CA\, with tracial rank zero and with a unique tracial state.
Suppose that $x, y\in A$ are two normal elements with ${\rm sp}(x)=X$ and ${\rm sp}(y)=Y.$
Denote by $j_x: C(X)\to A$ and $j_y: C(Y)\to A$
the embedding given by $j_x(f)=f(x)$ for all $f\in C(X)$ and $j_y(g)=g(y)$ for all $g\in C(Y).$
Suppose
that there exists a sequence of  unital  completely positive linear maps $\Phi_n: A\to A$ such that
\beqq
&&\lim_{n\to \infty}\|\Phi_n({{z_x}})-{{z_y}}\|=0\tand\\
&&{ {\lim_{n\to\infty}|\tau(\Phi_n(a))-\tau(a)|}}=0\tforal a\in A.
\eneqq
Then there exists a sequence of unital positive linear maps $\Psi: C(X)\to C(Y)$ such
that
\beq
&&\lim_{n\to \infty}\|\Psi_n(z_x)-{{y}}\|=0\tand\\
&&\lim_{n\to\infty}|\tau(\Psi_n(f)(y)))-\tau(f(x))|=0\tforal f\in C(X).
\eneq

\end{theorem}

\begin{proof}
Let $\ep>0,$  $\sigma>0$ and let ${\cal F}_X\subset C(X)$ be a finite subset {{in the unit ball.}}
Denote by $z_x$ the identify function on $X$ and $z_y$ the identity function on $Y.$
Put ${\cal F}_Y=\{1, z_y\}\subset C(Y).$
Let $\eta>0$ and ${\cal G}_Y\subset C(Y)$ be finite subset given by  for $\ep/2$ (in place of $\ep$)
and $\sigma/2$ (in place of $\sigma$), and $Y$ in place of $X.$  We may assume that ${\cal F}_Y\subset {\cal G}_Y.$

Put $\ep_1=\min\{\ep/16, \sigma/16, \eta/16\}.$ Let $\tau$ be the unique tracial state of $A.$

Let $\Phi: A\to A$ be a unital completely positive linear map
such that
\beqq\label{52-m-12}
&&\|\Phi(z_x)-y\|<\ep_1\andeqn\\\label{52-m-13}
&&|\tau(\Phi({{j_x(f))}})-\tau({{j_x(f)}})|<\ep_1\rforal f\in {\cal F}_X.
\eneqq
Let ${\cal F}_A={\cal F}_X\cup {\cal G}_Y\cup \Phi({\cal F}_X).$ \Wlog, we may assume
that ${\cal F}_A$ is in the unit ball of $A.$
Since $A$ has tracial rank zero, there is a finite dimensional \SCA\, $F\subset A$ with
$1_F=p$ and a \cpc\, $\psi: A\to F$  such that
\beqq
&&\|pa-ap\|<\ep_1\rforal a\in {\cal F}_A,\\
&&\|a-((1-p)a(1-p)+\psi(a))\|<\ep_1\rforal a\in {\cal F}_A\andeqn\\
&&\tau(1-p)<\ep_1.
\eneqq
\Wlog, by Lemma \ref{LLin}, we may assume that
there exists a unital \hm\, $\phi: C(Y)\to F$
such that
\beqq
\|\phi(g)-\psi(g)\|<\ep_1\rforal g\in {\cal G}_Y.
\eneqq
We also have
\beqq\label{52-m-5}
&&|\tau(\psi(a))-\tau(a)|<3\ep_1\rforal a\in {\cal F}_A\andeqn\\
&&|\tau(\phi(g))-\tau(j_y(g))|<4\ep_1\rforal g\in {\cal G}_Y.
\eneqq
Let $C=\phi(C(Y))$ be the \SCA\, of $F.$
By applying  Lemma \ref{2Lsbspliting}, we obtain a unital \cpc\, $L: F\to C(Y)$ such that
\beqq\label{52-m-10}
&&\|L\circ \phi(f)-j_y(f)\|<\ep/2\rforal f\in {\cal G}_Y\andeqn\\\label{52-m-11}
&&\hspace{-0.2in}|\tau\circ L(b)-\tau(b)|\le (\sigma/2)\|b\|
\rforal b\in F.
\eneqq

Let $S\subset F$ be a finite subset which is
$\ep_1$-dense in the unit ball of $F.$
Define $\Psi: C(X)\to C(Y)$ by
$\Psi(f)=L\circ \psi\circ \Phi(j_x(f))$ for $f\in C(X).$
Then, by \eqref{52-m-12} and \eqref{52-m-10},
\beqq
&&\|\Psi(z_x)-y\|\le \|L\circ \psi\circ \Phi(x)-L\circ \psi(z_y)\|+\|L\circ \psi(z_y)-y\|\\
&&<\|\Phi(x)-z_y\|+\ep/2<\ep_1+\ep/2<\ep.
\eneqq
Moreover, by \eqref{52-m-11}, \eqref{52-m-5}  and \eqref{52-m-13}
\beq
&&|\tau\circ \Psi(f)-\tau(j_x(f))|\le |\tau(L\circ \psi\circ \Phi(j_x(f)))-\tau(\psi\circ \Phi(j_x(f)))|\\
&&+
|\tau(\psi\circ \Phi(j_x(f)))-\tau(\Phi(j_x(f))))|
+|\tau(\Phi(j_x(f)))-\tau(j_x(f))|\\
&&\le \sigma/2+3\ep_1+\ep_1<\sigma\rforal f\in {\cal F}_X.
\eneq

\end{proof}

Let $X$ be a compact metric space. Denote by
$M(X)^{\bf 1}$ the set of all probability Borel measures.

Let $A$ be a unital simple \CA\, with $T(A)\not=\emptyset,$  and let $x\in A$ be normal element.
Denote by $j_x: C(X)\to A$ the embedding.
For each $\tau\in T(A),$ denote by $\mu_{\tau, X}$ (or just $\mu_\tau$) the  probability Borel measure induced by
$\tau\circ j_x.$
Define $T_X=\{\mu_{\tau, X}: \tau\in T(A)\}.$

\begin{theorem}\label{T5}
Let $A$ be a unital separable simple \CA\, with tracial rank zero
and with  unique tracial state $\tau.$
Suppose
that $x$ and $y$ are two normal elements
with $X={\rm sp}(x)$ and $Y={\rm sp}(y).$
Let $z_x\in C(X)$ be the identity function on $X$ and
let $z_y\in C(Y)$  be the identity function on $Y.$
Then the following are equivalent:

(1) $y\in \overline{\conv({\cal U}(x))};$

(2)  There exists a sequence of unital trace preserving  completely positive linear
maps $\Phi_n: A\to A$ such that
\beq
\lim_{n\to\infty}\|\Phi_n(x)-y\|=0;
\eneq

(3) There exists a sequence of unital completely positive linear maps
$\Phi_n: A\to A$ such that
\beq
&&\lim_{n\to\infty}\|\Phi_n(x)-y\|=0\tand\\
&&\lim_{n\to\infty}|\tau(\Phi_n(a))-\tau(a)|=0\tforal a\in A;
\eneq

(4) There exists a sequence of unital completely positive linear maps
$\Psi_n: C(X)\to C(Y)$ such that
\beqq\label{5f4-1}
&&\lim_{n\to\infty}\|\Psi_n(z_x)-z_y\|=0\tand\\\label{5f4-2}
&&\lim_{n\to\infty}\Big|\tau(\Psi_n(f)(y)))-\tau(f(x))\Big|=0\tforal f\in C(X);
\eneqq

(5) There exists a sequence of affine continuous maps $\gamma_n: M(Y)^{\bf 1} \to M(X)^{\bf 1}$
such that
\beqq\label{5f5-1}
&&\lim_{n\to\infty}\sup \{|\int_X x d(\gamma_n(\mu))-\int_Y y d\mu|: \mu\in M(Y)^{\bf 1}\}=0\andeqn\\\label{5f5-2}
&&\lim_{n\to\infty}\Big|\int_X f d\mu_{\tau, X}-\int_X f d \gamma_n(\mu_{\tau, Y})\Big|=0
\eneqq
for all $f\in C(X).$


\end{theorem}

\begin{proof}
That (1), (2) and (3) are equivalent follows from \ref{beta}.
That  (3) implies (4) follows from \ref{52-m}.

Suppose that (4) holds. Define $\gamma_n: M(Y)^{\bf 1}\to M(X)^{\bf 1}$
by $\int_X f d(\gamma_n(\mu))=\int_Y \Psi_n(f) d\mu$ for all $\mu\in M(Y)^{\bf 1}.$
Clearly $\gamma_n$ is a continuous affine map.
Denote by $S(C(Y))$ the state space of $C(Y).$
Then,  by \eqref{5f4-1},
\beq
\lim_{n\to\infty}\sup\{|s(\Psi_n(x))-s(z_y)|: s\in S(C(Y))\}=0,
\eneq
Since one may identify $S(C(Y))$ with $M(Y)^{\bf 1},$ \eqref{5f5-1} follows.
It is also clear that \eqref{5f5-2} follows from \eqref{5f4-2}.  Thus (5) holds.

We now show that (5) implies (4).
If (5) holds, define
$\Psi_n(f)(s)=\gamma_n(s)(f)$ for all $s\in S(C(Y))=M(Y)^{\bf 1}$ and any $f\in C(X).$
Then \eqref{5f5-1} implies that
\beq
\lim_{n\to\infty}\sup\{|\Psi_n(z_x)(s)-z_y(s)|: s\in S(C(Y))\}=0.
\eneq
However,
\beq
\sup\{|\Psi_n(z_x)(s)-z_y(s)|: s\in S(C(Y))\}=\|\Psi_n(z_x)-y_z\|.
\eneq
It follows that \eqref{5f4-1} holds.  Also \eqref{5f4-2} follows from \eqref{5f5-2}.

It remains to show that (4) implies (1) which follows from \ref{52cc}.

\end{proof}

\begin{corollary}\label{CC5end}
Let $A$ be a unital separable simple \CA\, with real rank zero, stable rank one, weakly unperforated
$K_0(A)$
and with  unique  quasi-trace $\tau$ such that  $\tau(1_A)=1.$
Suppose
that $x$ and $y$ are two normal elements
with $X={\rm sp}(x)$ and $Y={\rm sp}(y).$
Then  (1), (4) and (5) in \ref{T5} are  also equivalent,
by replacing the tracial state by the quasi-trace.
\end{corollary}

\begin{proof}
Note that any quasi-trace restricted on a commutative \CA\, is a trace.
We also note that (4) and (5) are equivalent.
It remains to show (1) and (4) are equivalent.
We deploy the argument of the proof of \ref{Trr0}.
We keep all notation there.

If (1) holds, then $x_1\in \overline{\conv}({\cal U}(y_1)).$
Thus \ref{T5} can apply to $x_1$ (in place $x$)  and $y_1$ (in place of $y$) (in $B$).
Since $\sp(x)=\sp(x_1)$ and $\sp(y)=\sp(y_1),$ as functions
in $C(X)$ and $C(Y),$ $z_x=z_{x_1}$ and $z_y=z_{y_1},$ respectively.
Therefore \eqref{5f4-1} holds.
Since $x_1$ and $x$ are approximately unitarily equivalent, $\tau(f(x_1))=\tau(f(x))$ for all $f\in C(X),$
and $\tau(g(y))=\tau(g(y_1))$ for all $g\in C(Y).$ Hence
\eqref{5f4-2} also holds.

Suppose (4) holds.
Then there exists a sequence of unital completely positive linear maps
$\Psi_n: C(X)\to C(Y)$ such that
\beqq\label{5f4-1_n}
&&\lim_{n\to\infty}\|\Psi_n(z_x)-z_y\|=0\tand\\\label{5f4-2+1}
&&\lim_{n\to\infty}\Big|\tau(\Psi_n(f)(y))-\tau(f(x)) \Big|=0\tforal f\in C(X).
\eneqq

 The same reason given above shows  (4) holds for
$x_1$ (in place of $x$) and $y_1$ (in place of $y$) in $B.$
Therefore $y_1\in \overline{\conv({\cal U}(x_1))}.$ It follows that $y\in \overline{\conv({\cal U}(x))}.$
So (1) holds.

\end{proof}

\section{Approximately unitarily equivalence}
In what follows, if $x$ is a normal element in a  unital \CA,
we denote by $j_x: C({\rm sp}(x))\to A$ the injective \hm\, defined by
$j_x(f)=f(x)$ for all $f\in C({\rm sp}(x)).$

Let $A$ be a unital simple \CA\,  with tracial rank zero and
$x, \, y\in A$ be two normal elements.  By  \cite{LinCHD},
$x$ and $y$ are approximately unitarily equivalent, i.e., there exists a
sequence of unitaries $\{u_n\}$ of $A$ such that
\beq
\lim_{n\to\infty}\| u_n^*xu_n-y\|=0,
\eneq
if and only if $(j_x)_{*i}=(j_y)_{*i},$ $i=0,1,$ and,
$\tau\circ j_x=\tau\circ j_y$ for all $\tau\in T(A)$ (see also 5.6 of \cite{HL} for a slightly 
more general setting of this statement).

\begin{theorem}\label{T51}
Let $A$ be a unital  separable simple \CA\, with real rank zero, stable rank one,
weakly unperforated $K_0(A)$ and with a unique quasi-trace $\tau$ with $\tau(1_A)=1.$
and let
$x, y\in A$ be two normal elements.
Then the following are equivalent:

{\rm (1)} $x\in \overline{\conv({\cal U}(y))}$  and $y\in \overline{\conv({\cal U}(x))};$

{\rm (2)} ${\rm sp}(x)={\rm sp}(y)$ and  $\mu_x=\mu_y,$
where $\mu_x$ and $\mu_y$ are Borel probability measures
induced by $\tau\circ j_x$ and $\tau\circ j_y,$ respectively,

\end{theorem}

\begin{proof}
Suppose that (2) holds.  Let $\phi: C({\rm sp}(x))\to C({\rm sp}(y))$
be defined by $\phi(f)=f(y)$ for all $f\in C({\rm sp}(y)).$ Then
(1) follows by \ref{CC5end}.

Suppose that (1) holds.
Let $\pi_\tau: A\to B(H_\tau)$ be the representation of $A$ given by the tracial state $\tau.$
Let $M=\pi_\tau(A)''.$ Then $M$ is a type {II}$_1$ factor.   Note that
since $A$ is simple, $\pi_\tau$ is faithful.
Note (1) implies that
\beq\nonumber
\pi_\tau(x)\in \overline{\conv({\cal U}(\pi_\tau(y)))}\andeqn \pi_\tau(y)\in \overline{\conv({\cal U}(\pi_\tau(x)))}.
\eneq
By  (vi) of Theorem 2.2 of \cite{H1},  $\tau(g(x))=\tau(g(y))$  for all
continuous convex function on $\R^2.$  It follows from Proposition I.1.1 of \cite{Alf}
that
\beq
\tau(f(x))=\tau(f(\pi_\tau(x)))=\tau(f(\pi_\tau(y)))=\tau(f(y))
\eneq
for all $f\in C({\rm sp}(x)).$  Consequently
${\rm sp}(x)={\rm sp}(\pi_\tau(x))={\rm sp}(\pi_\tau(y))={\rm sp}(y).$
Thus (2) holds.

\end{proof}

Let $A$ be a \CA.  Denote by $\rho_A: K_0(A)\to \Aff(T(A))$  the usual order preserving \hm.

\begin{corollary}\label{C52}
Let $A$ be a unital   separable simple \CA\, with  real rank zero, stable rank one, weakly unperforated
$K_0(A)$ and with a unique quasi-trace
$\tau$ such that $\tau(1_A)=1,$  and let
$x, y\in A$ be two normal elements.
Suppose that $K_1(A)=\{0\}$ and ${\rm ker}\rho_A=\{0\}.$
Then the following are equivalent:

{\rm (1)} $x\in\in \overline{\conv({\cal U}(y))}$  and $y\in \overline{\conv({\cal U}(x))};$

{\rm (2)} ${\rm sp}(x)={\rm sp}(y)$ and
$\tau(f(x))=\tau(f(y))$ for all $f\in C({\rm sp}(x)).$

{\rm (3)} $x$ and $y$ are approximately unitarily equivalent in $A.$

\end{corollary}

\begin{proof}
It is clear that (3) implies (1). Thus, by \ref{T51}, it remains to show
that (2) implies (3).   Assume that (2) holds. By the assumption, $\rho_A: K_0(A)\to \Aff(T(A))$
is an order preserving injective \hm. Therefore, (2), together with the assumption
that $K_1(A)=\{0\},$ implies that $(j_x)_{*i}=(j_y)_{*i},$ $i=0,1,$
where $j_x$ and $j_y$ are embedding from
$C({\rm sp}(x))$ to $A$ induced by $x$ and $y,$ respectively.
Since $K_0(C({\rm sp}(x)))=C({\rm sp}(x), \Z)$ is a free abelian group (see \cite{No}),
it follows from the Universal Coefficient Theorem that $[j_x]=[j_y]$ in $KL(C({\rm sp}(x)), A).$
Then,  by 5.6 of \cite{HL},
$j_x$ and $j_y$ are approximately unitarily equivalent, whence
(3) holds.

\end{proof}

Let $A$ be a unital simple \CA\, of tracial rank zero such that $K_1(A)\not=\{0\}.$
It follows from Theorem 6.11 of \cite{LinAH} that there are two
normal elements $x, y$ such that (2) of \ref{C52} holds but $(j_x)_{*1}\not=(j_y)_{*1}.$
Then $x$ and $y$ are not approximately unitarily equivalent.
However,  by \ref{C52}, $x\in \overline{\conv}(\uu(y))$ and $y\in \overline{\conv}(\uu(x)).$
Suppose that $K_1(A)=\{0\}$ but ${\rm ker}\rho_A\not=\{0\}.$   Suppose
$X\subset \C$ is  compact subset which is not connected. Then, by 6.11 of \cite{LinAH} again, there are
normal elements $x, y\in A$ with $\sp(x)=\sp(y)=X$ such that (2) of \ref{C52} holds but
$(j_x)_{*0}\not=(j_y)_{*0}.$ Then $x$ and $y$ are not approximately unitarily equivalent.
However, by \ref{C52} again,   $x\in \overline{\conv}(\uu(y))$ and $y\in \overline{\conv}(\uu(x)).$
Nevertheless, we have the following:

\begin{corollary}\label{C53}
Let $A$ be a unital
separable simple AF-algebra with a
unique tracial state  and let
$x,\, y\in A$ be two normal elements with connected spectrum.

Then the following are equivalent:

{\rm (1)} $x\in\overline{\conv({\cal U}(y))}$  and $y\in \overline{\conv({\cal U}(x))};$

{\rm (2)} ${\rm sp}(x)={\rm sp}(y)$ and
$\tau(f(x))=\tau(f(y))$ for all $f\in C({\rm sp}(x)).$

{\rm (3)} $x$ and $y$ are approximately unitarily equivalent in $A.$

\end{corollary}

\begin{proof}
Again, it remains to show  (2) implies (3).
Since both ${\rm sp}(x)$ and ${\rm sp}(y)$ are connected,
(2) implies that $(j_x)_{*i}=(j_y)_{*i},$ $i=0,1.$  Then, by 3.4  of \cite{LinCHD},  as
in the proof of \ref{C52}, (3) holds.

\end{proof}

\end{document}